\documentclass[sn-mathphys]{sn-jnl}

\usepackage{color}
\usepackage{epic}

\jyear{2021}%

\theoremstyle{thmstyleone}%
\newtheorem{theorem}{Theorem}%
\newtheorem{lemma}{Lemma}%

\theoremstyle{thmstyletwo}%
\newtheorem{example}{Example}%
\newtheorem{remark}{Remark}%

\theoremstyle{thmstylethree}%

\def\rhs{g}
\newcommand\ba{{\boldsymbol a}}
\newcommand\bb{{\boldsymbol b}}
\newcommand\RR{\mathbb{R}}
\newcommand\TT{{\mathscr T}}
\newcommand\II{{\mathscr I}}
\newcommand\NN{{\mathscr N}}

\begin{document}

\title[An algebraically stabilized method for convection--diffusion--reaction
problems with optimal experimental convergence rates on general meshes]
{An algebraically stabilized method for convection--diffusion--reaction
problems with optimal experimental convergence rates on general meshes}

\author{\fnm{Petr} \sur{Knobloch}}\email{knobloch@karlin.mff.cuni.cz, ORCID 0000-0003-2709-5882}

\affil{\orgdiv{Department of Numerical Mathematics, Faculty of Mathematics
and Physics}, \orgname{Charles University},
\orgaddress{\street{Sokolovsk\'a 83}, \city{Praha 8}, \postcode{18675},
\country{Czech Republic}}}

\abstract{Algebraically stabilized finite element discretizations of scalar
steady-state convection--diffusion--reaction equations often provide accurate 
approximate solutions satisfying the discrete maximum principle (DMP). However,
it was observed that a deterioration of the accuracy and convergence rates may 
occur for some problems if meshes without local symmetries are used. The paper
investigates these phenomena both numerically and analytically and the findings
are used to design a new algebraic stabilization called Symmetrized 
Monotone Upwind-type Algebraically Stabilized (SMUAS) method. It is proved that
the SMUAS method is linearity preserving and satisfies the DMP on arbitrary 
simplicial meshes. Numerical results indicate that the SMUAS method leads to
optimal convergence rates on general meshes.}


\maketitle

This work has been
supported through the grant No.~22-01591S of the Czech Science Foundation.

\pagestyle{myheadings}
\thispagestyle{plain}
\markboth{P.~Knobloch}
{An algebraically stabilized method with optimal convergence rates}

\newpage

\section{Introduction}
\label{s1}

Convection, diffusion and reaction are basic physical mechanisms which play an
important role in many mathematical models used in science and technology. A
frequently used model problem for studying numerical techniques for the
mentioned class of models is the scalar steady-state 
convection--diffusion--reaction problem
\begin{equation}
   -\varepsilon\,\Delta u+{\bb}\cdot\nabla u+c\,u=\rhs\quad\mbox{in $\Omega$}\,,
    \qquad\qquad u=u_b\quad\mbox{on $\partial\Omega$}\,,\label{strong-steady}
\end{equation}
where $\Omega\subset{\mathbb R}^d$, $d\ge1$, is a bounded domain,
$\varepsilon>0$ is a constant diffusion coefficient, $\bb$ is the convection
field, $c$ is the reaction field, and the right-hand side $\rhs$ is a source of
the unknown quantity $u$. Note that the model problem \eqref{strong-steady} 
itself has also a clear physical meaning since it may describe, e.g., the 
distribution of temperature or concentration. For our mathematical 
considerations, we will assume that the boundary $\partial\Omega$ of $\Omega$ 
is polyhedral and Lipschitz-continuous (if $d\ge2$) and that 
${\bb}\in W^{1,\infty}(\Omega)^d$, $c\in L^\infty(\Omega)$, 
$\rhs\in L^2(\Omega)$, and 
$u_b\in H^{\frac{1}{2}}(\partial\Omega)\cap C(\partial\Omega)$. Moreover, it 
will be assumed that the data satisfy the conditions
\begin{equation}
   \nabla\cdot\bb=0\,,\qquad c\ge\sigma_0\ge0\qquad\quad\mbox{in $\Omega$}\,,
   \label{eq_ass_b_c}
\end{equation}
where $\sigma_0$ is a constant. 

In most applications, the convective transport strongly dominates the diffusion
which causes that the solution $u$ comprises so-called layers, which are narrow
regions where $u$ changes abruptly. The presence of layers makes the numerical
solution of \eqref{strong-steady} very challenging since standard approaches
provide solutions polluted by spurious oscillations unless the layers are
resolved by the mesh. A well-known remedy is a stabilization of the standard
discretization, e.g., by adding additional stabilization terms, see, e.g.,
\cite{RST08}. To obtain accurate approximations, the stabilization has to be
adopted to the character of the approximated solution which inevitably leads to
nonlinear methods. However, many of such stabilization techniques still do not
remove the spurious oscillations completely since the stabilization effect is
influenced by many factors, like the used mesh or the considered data,
cf.~\cite{JK07,JK07b,JK08}. Although the remaining spurious oscillations are 
often quite small, they may be not acceptable in some applications, e.g., if 
the oscillating solution should serve as input data for other equations. A
possible remedy is to apply methods satisfying the discrete maximum 
principle (DMP), see, e.g., the recent review paper \cite{BJK22}. The DMP 
excludes many types of oscillating solutions that otherwise frequently appear
when solving convection-dominated problems. A further reason for 
requiring the validity of the DMP is that a maximum principle holds for the 
continuous problem \eqref{strong-steady} if $c\ge0$ (cf.~\cite{Evans,GT01}) and 
it is important that this physical property is preserved by the discrete 
problem. 

An interesting class of methods satisfying the DMP (often under some
assumptions on the mesh) are algebraically stabilized finite element schemes, 
e.g., algebraic flux correction (AFC) schemes. These methods have been 
developed intensively in recent years, see, e.g.,
\cite{BB17,BJK17,GNPY14,Kuzmin06,Kuzmin07,Kuzmin09,Kuzmin12,KuzminMoeller05,KS17,KT04,LKSM17}. The origins of this approach can be tracked back to 
\cite{BorisBook73,Zalesak79}. In these schemes, the stabilization is performed 
on the basis of the algebraic system of equations corresponding to the Galerkin 
finite element method. It involves so-called limiters, which restrict the 
stabilized discretization mainly to a vicinity of layers to ensure the 
satisfaction of the DMP without compromising the accuracy. There are several 
limiters proposed in the literature, like the so-called Kuzmin~\cite{Kuzmin07}, 
BJK~\cite{BJK17}, or BBK~\cite{BBK17b} limiters. Both, the Kuzmin and the BBK 
limiters were utilized in \cite{BBK17a} for defining a scheme that blends a 
standard linear stabilized scheme in smooth regions and a nonlinear stabilized 
method in a vicinity of layers. 

An important feature of algebraically stabilized schemes is that they not only 
satisfy the DMP but also usually provide sharp approximations of layers, 
cf.~the numerical results in, e.g., \cite{ACF+11,GKT12,JS08,Kuzmin12}. In this
paper, we concentrate on schemes based on the idea of algebraic flux 
correction. Many properties of the AFC schemes are already well understood 
since these schemes were investigated in a number of papers, see, e.g., 
\cite{BJK15,Kno15b,BJK16,BJK17,Kno17,BJKR18,Kno19}, where one can find results 
on the existence of solutions, local and global DMPs, error estimates, and 
further properties. However, it was observed already in \cite{BJK16} that
convergence rates of these schemes may be suboptimal on some meshes, even if
problems without layers are considered. The aim of the present paper is to
explain this behaviour in some model cases and, using the results of this
analysis, to propose modifications of the considered methods leading to optimal 
convergence rates. This will lead to a new algebraic stabilization called 
Symmetrized Monotone Upwind-type Algebraically Stabilized (SMUAS) method for
which the solvability, linearity preservation and DMP will be proved on 
arbitrary simplicial meshes. Moreover, various numerical results will be 
reported that show that, in many cases, the SMUAS method leads to more accurate 
results than other algebraic stabilizations. In addition, the numerical results 
indicate that the SMUAS method converges with optimal rates on general meshes. 
Let us mention that the analysis of AFC schemes also demonstrates the 
interesting fact that certain types of spurious oscillations may be still 
present in the approximate solutions despite the validity of the DMP. This 
contradicts the frequently made claim that the DMP guarantees that no spurious 
oscillations appear.

The plan of the paper is as follows. In the next section, we define a Galerkin 
finite element discretization of \eqref{strong-steady} and the corresponding 
linear algebraic problem. Then, in Section~\ref{s3}, we introduce a general 
algebraic stabilization and summarize its main properties. Section~\ref{s4} 
provides three examples of algebraic stabilizations. The first one is the AFC 
scheme with the Kuzmin limiter, the deficiencies of which are then analyzed in
Section~\ref{s5}. The other two examples in Section~\ref{s4} are the AFC scheme
with the BJK limiter, for which also some results are reported in
Section~\ref{s5}, and the MUAS method. The MUAS method is used as the basis for
defining the new algebraic stabilization in Section~\ref{s6}. After analyzing
the new method, various numerical results will be presented.

\section{Galerkin finite element discretization}
\label{s2}

A finite element discretization of the convection--diffusion--reaction problem 
\eqref{strong-steady} is based on its weak formulation, which reads: 

\vspace*{2mm}

Find $u\in H^1(\Omega)$ such that $u=u_b$ on $\partial\Omega$ and
\begin{equation*}
   a(u,v)=(\rhs,v)\qquad\forall\,\,v\in H^1_0(\Omega)\,,
\end{equation*}
where
\begin{equation*}
   a(u,v)=\varepsilon\,(\nabla u,\nabla v)+({\bb}\cdot\nabla u,v)+(c\,u,v)\,.
\end{equation*}

\vspace*{1mm}

\noindent
As usual, $(\cdot,\cdot)$ denotes the inner product in $L^2(\Omega)$ or 
$L^2(\Omega)^d$. It is well known that this weak formulation has a unique
solution (cf. \cite{Evans}).

To define a finite element discretization of problem \eqref{strong-steady}, we
consider a simplicial triangulation $\TT_h$ of $\overline\Omega$ which is 
assumed to belong to a regular family of triangulations in the sense of
\cite{Ciarlet}. Furthermore, we introduce finite element spaces
\begin{displaymath}
   W_h=\{v_h\in C(\overline{\Omega})\,;\,\,v_h\vert _T^{}\in P_1(T)\,\,
   \forall\, T\in \TT_h\}\,,\qquad V_h=W_h\cap H^1_0(\Omega)\,,
\end{displaymath}
consisting of continuous piecewise linear functions. The vertices of the
triangulation $\TT_h$ will be denoted by $x_1,\dots,x_N$ and we assume that
$x_1,\dots,x_M\in\Omega$ and $x_{M+1},\dots,x_N\in\partial\Omega$. Then the
usual basis functions $\varphi_1,\dots,\varphi_N$ of $W_h$ are defined by the
conditions $\varphi_i(x_j)=\delta_{ij}$, $i,j=1,\dots,N$, where $\delta_{ij}$
is the Kronecker symbol. Obviously, the functions $\varphi_1,\dots,\varphi_M$ 
form a basis of $V_h$. Any function $u_h\in W_h$ can be written in a unique way
in the form
\begin{equation}
   u_h=\sum_{i=1}^Nu_i\,\varphi_i\label{eq_vect_fcn_identification}
\end{equation}
and hence it can be identified with the coefficient vector 
${\rm U}=(u_1,\dots,u_N)$.

Now an approximate solution of problem \eqref{strong-steady} can be
introduced as the solution of the following finite-dimensional problem:

\vspace*{2mm}

Find $u_h\in W_h$ such that $u_h(x_i)=u_b(x_i)$, $i=M+1,\dots,N$, and
\begin{equation}
   a(u_h,v_h)=(\rhs,v_h)\qquad\forall\,\,v_h\in V_h\,,\label{v3}
\end{equation}

\vspace*{1mm}

\noindent
It is easy to show that the discrete problem \eqref{v3} has a unique solution. 

We denote
\begin{alignat}{2}
   a_{ij}&=a(\varphi_j,\varphi_i)\,,\qquad &&i,j=1,\dots,N\,,\label{13}\\
   \rhs_i&=(\rhs,\varphi_i)\,,\qquad &&i=1,\dots,M\,,\label{14}\\
   u^b_i&=u_b(x_i)\,,\qquad &&i=M+1,\dots,N\,.\label{15}
\end{alignat}
Then $u_h$ is a solution of the finite-dimensional problem \eqref{v3} if and 
only if the coefficient vector $(u_1,\dots,u_N)$ corresponding to $u_h$ 
satisfies the algebraic problem
\begin{align}
   &\sum_{j=1}^N\,a_{ij}\,u_j=\rhs_i\,,\qquad i=1,\dots,M\,,\label{21}\\
   &u_i=u^b_i\,,\qquad i=M+1,\dots,N\,.\label{21b}
\end{align}

As discussed in the introduction, the above discretization is not appropriate 
in the convection-dominated regime and a stabilization has to be applied. The
most common way is to introduce additional stabilization terms in the discrete 
problem \eqref{v3}, see, e.g., \cite{RST08}. However, another attractive 
possibility is to modify the algebraic problem \eqref{21}, \eqref{21b}, which
will be pursued in this paper.

\section{A general algebraic stabilization}
\label{s3}

The stabilizing effect of various approaches used to suppress the spurious
oscillations present in the solutions of the Galerkin discretization is due to
the fact that these methods add a certain amount of artificial diffusion to the 
Galerkin FEM. However, if this amount is too large, the approximate solution
becomes inaccurate due to an excessive smearing of the layers. It turns out
that accurate solutions can be obtained only if the amount of the artificial 
diffusion respects the local behaviour of the solution, see, e.g.,
\cite{BJK22}. This motivates us to stabilize the algebraic problem
\eqref{21}, \eqref{21b} by adding an artificial diffusion matrix
${\mathbb B}({\rm U})=(b_{ij}({\rm U}))_{i,j=1}^N$ which depends on the unknown 
approximate solution ${\rm U}=(u_1,\dots,u_N)$. Here we shall describe this
approach only briefly and refer to the recent paper \cite{JK21} for a more
detailed presentation.

Based on the above discussion, we will consider the nonlinear algebraic problem
\begin{align}
   &\sum_{j=1}^N\,(a_{ij}+b_{ij}({\rm U}))\,u_j=\rhs_i\,,\qquad i=1,\dots,M\,,
   \label{23}\\
   &u_i=u^b_i\,,\qquad i=M+1,\dots,N\,.\label{23b}
\end{align}
We assume that, for any ${\rm U}\in{\mathbb R}^N$, the matrix ${\mathbb B}({\rm
U})$ satisfies
\begin{alignat}{2}
   &b_{ij}({\rm U})=b_{ji}({\rm U})\,,\qquad\quad&&i,j=1,\dots,N\,,
   \label{eq_b1}\\[2mm]
   &b_{ij}({\rm U})\le0,\qquad\quad&&i,j=1,\dots,N\,,\,\,i\neq j\,,
   \label{eq_b2}\\
   &\sum_{j=1}^N\,b_{ij}({\rm U})=0\,,\qquad\quad&&i=1,\dots,N\,.\label{eq_b3}
\end{alignat}
Moreover, we assume that ${\mathbb B}({\rm U})$ has the typical sparsity 
pattern of finite element matrices, i.e.,
\begin{equation}
   b_{ij}({\rm U})=0\qquad
   \forall\,\,j\not\in S_i\cup\{i\},\,i=1,\dots,M\,,\label{eq_b4}
\end{equation}
where
\begin{equation*}
   S_i=\{j\in\{1,\dots,N\}\setminus\{i\}\,;\,\,\mbox{\rm $x_i$ and $x_j$ are
         end points of the same edge}\}\,.
\end{equation*}
These assumptions are motivated by the fact that the properties
\eqref{eq_b1}--\eqref{eq_b4} are satisfied for the diffusion matrix 
$(\varepsilon\,(\nabla\varphi_j,\nabla\varphi_i))_{i,j=1}^N$ 
if the triangulation $\TT_h$ is weakly acute, i.e., 
if the angles between facets of $\TT_h$ do not exceed $\pi/2$. It is also
important that the properties \eqref{eq_b1}--\eqref{eq_b3} assure that the 
matrix ${\mathbb B}({\rm U})$ is positive semidefinite for any 
${\rm U}\in{\mathbb R}^N$, see \cite{JK21}.

To prove the solvability of the system \eqref{23}, \eqref{23b}, we make the
following assumption, which is motivated by the definitions of the matrix 
${\mathbb B}({\rm U})$ considered in this paper.

\vspace*{1ex}
\noindent{\bf Assumption (A1): } For any $i\in\{1,\dots,M\}$ and any
$j\in\{1,\dots,N\}$, the function $b_{ij}({\rm U})(u_j-u_i)$ is a continuous 
function of ${\rm U}=(u_1,\dots,u_N)\in{\mathbb R}^N$ and, for any
$i\in\{1,\dots,M\}$ and any $j\in\{M+1,\dots,N\}$, the function 
$b_{ij}({\rm U})$ is a bounded function of ${\rm U}\in{\mathbb R}^N$.
\vspace*{1ex}

\begin{theorem}\label{existence}
Let \eqref{eq_b1}--\eqref{eq_b3} hold and let Assumption (A1)
be satisfied. Then there exists a solution of the nonlinear problem
\eqref{23}, \eqref{23b}.
\end{theorem}

\begin{proof} See \cite{JK21}. \end{proof}

The construction of the matrix ${\mathbb B}({\rm U})$ is usually based on the
requirement that the problem \eqref{23}, \eqref{23b} satisfies the DMP.
One can formulate various conditions that guarantee that a nonlinear
discrete problem satisfies the DMP or at least preserves the positivity,
cf.~\cite{BJK22}. For our purposes, the following assumption is useful.

\vspace*{1ex}
\noindent{\bf Assumption (A2): } Consider any ${\rm U}=(u_1,\dots,u_N)\in\RR^N$ 
and any $i\in\{1,\dots,M\}$. If $u_i$ is a strict local extremum of $\rm U$ 
with respect to $S_i$, i.e.,
\begin{displaymath}
   u_i>u_j\quad\forall\,\,j\in S_i\qquad\mbox{or}\qquad
   u_i<u_j\quad\forall\,\,j\in S_i\,,
\end{displaymath}
then
\begin{displaymath}
   a_{ij}+b_{ij}({\rm U})\le0\qquad\forall\,\,j\in S_i\,.
\end{displaymath}
\vspace*{1ex}

Under the above assumptions, it is possible to prove that the approximate
solution obtained using the nonlinear problem \eqref{23}, \eqref{23b} satisfies
a direct analogue of the maximum principles which hold for the 
problem \eqref{strong-steady} (see, e.g., \cite{Evans} for the classical 
solutions and \cite{GT01} for the weak solutions).

\begin{theorem}\label{thm:general_DMP2}
Let the assumptions stated in Section~\ref{s1} be satisfied and let the matrix 
${\mathbb B}({\rm U})$ satisfies \eqref{eq_b1}--\eqref{eq_b4} and Assumptions 
(A1) and (A2). Consider any nonempty set ${\mathscr G}_h\subset \TT_h$ and 
define
\begin{displaymath}
   G_h=\bigcup_{T\in {\mathscr G}_h}\,T\,.
\end{displaymath}
Let ${\rm U}\in{\mathbb R}^N$ be a solution of \eqref{23} and let $u_h\in W_h$
be the corresponding finite element function given by
\eqref{eq_vect_fcn_identification}. Then one has the DMP
\begin{align*}
   &\rhs\le0\quad\mbox{\rm in}\,\,\,G_h\qquad\Rightarrow\qquad
    \max_{G_h}\,u_h\le\max_{\partial G_h}\,u_h^+\,,\\
   &\rhs\ge0\quad\mbox{\rm in}\,\,\,G_h\qquad\Rightarrow\qquad
    \min_{G_h}\,u_h\ge\min_{\partial G_h}\,u_h^-\,,
\end{align*}
where $u_h^+=\max\{u_h,0\}$ and $u_h^-=\min\{u_h,0\}$. If, in addition, $c=0$ 
in $G_h$, then
\begin{align*}
   &\rhs\le0\quad\mbox{\rm in}\,\,\,G_h\qquad\Rightarrow\qquad
    \max_{G_h}\,u_h=\max_{\partial G_h}\,u_h\,,\\
   &\rhs\ge0\quad\mbox{\rm in}\,\,\,G_h\qquad\Rightarrow\qquad
    \min_{G_h}\,u_h=\min_{\partial G_h}\,u_h\,.
\end{align*}
\end{theorem}

\begin{proof} See \cite{JK21}. \end{proof}

We will close this section with a brief discussion of a priori error estimates 
available for the nonlinear problem \eqref{23}, \eqref{23b}. To derive an error
estimate, it is convenient to write \eqref{23}, \eqref{23b} as a variational
problem where the algebraic stabilization term is represented using the form
\begin{displaymath}
   b_h(w;z,v)
   =\sum_{i,j=1}^N\,b_{ij}(w)\,z(x_j)\,v(x_i)\,,
   \qquad w,z,v\in C(\overline\Omega)\,,
\end{displaymath}
with $b_{ij}(w):=b_{ij}(\{w(x_i)\}_{i=1}^N)$, see \cite{JK21} for details. This
variational problem is stable with respect to the solution-dependent norm 
on $W_h$ defined by
\begin{displaymath}
   \| v\| _h^{}:=\Big(\varepsilon\,\vert v\vert _{1,\Omega}^2
   +\sigma_0\,\| v\| _{0,\Omega}^2
   +b_h(u_h;v,v)\Big)^{1/2}\,,\qquad v\in H^1(\Omega)\cap C(\overline\Omega)\,,
\end{displaymath}
assuming that $\sigma_0>0$ in \eqref{eq_ass_b_c}. This shows that the problem
\eqref{23}, \eqref{23b} really provides a stronger stability than the original
problem \eqref{21}, \eqref{21b}. 

The algebraic stabilization term leads to a consistency error whose behaviour 
with respect to $h$ depends on how the artificial diffusion matrix is 
constructed. Often, one has 
\begin{equation*}
   \vert b_{ij}(u_h)\vert \le\max\{\vert a_{ij}\vert ,\vert a_{ji}\vert \}\qquad\forall\,\,i\neq j\,,
\end{equation*}
which will be also the case in this paper. Under this assumption and assuming
further that the weak solution of \eqref{strong-steady} satisfies 
$u\in H^2(\Omega)$ and that $\sigma_0>0$, one can prove (cf.~\cite{JK21}) that 
the finite element function $u_h\in W_h$, corresponding via
\eqref{eq_vect_fcn_identification} to the solution ${\rm U}\in{\mathbb R}^N$ of 
the nonlinear algebraic problem \eqref{23}, \eqref{23b}, satisfies the estimate
\begin{align}
   \| u-u_h\| _h^{}&\le 
   C\,(\varepsilon+\sigma_0^{-1}\,\{\| \bb\| _{0,\infty,\Omega}^2
   +\| c\| _{0,\infty,\Omega}^2\}+\sigma_0h^2)^{1/2}\,h\,\| u\| _{2,\Omega}
   \nonumber\\
   &\hspace*{16mm}+ C\,(\varepsilon+\| \bb\| _{0,\infty,\Omega}\,h
   +\| c\| _{0,\infty,\Omega}^{}\,h^2)^{1/2}\,\vert i_hu\vert _{1,\Omega}\,,
   \label{eq:error_est}
\end{align}
where the constant $C$ is independent of $h$ and the data of problem 
\eqref{strong-steady}. If $\sigma_0=0$, then the estimate is deteriorated by a
negative power of $\varepsilon$, see \cite{BJK16} for details. We also refer
to \cite{BJK16} and \cite{BJKR18} for slightly improved error estimates under 
various additional assumptions.

The estimate \eqref{eq:error_est} does not imply any convergence in the
diffusion-dominated case (when $\varepsilon>\| \bb\| _{0,\infty,\Omega}\,h$) 
and it guarantees only the convergence order $1/2$ in the convection-dominated 
case. Numerical results presented in \cite{BJK16} show
that this result is sharp under the general assumptions made up to now. It is 
of course
desirable to design the artificial diffusion matrix ${\mathbb B}({\rm U})$ in
such a way that optimal convergence rates with respect to various norms are
obtained. For some algebraic stabilizations, optimal convergence rates were
indeed observed, however, a more detailed convergence studies revealed that the
convergence rates often depend on the considered meshes and data,
cf.~\cite{BJK16,BJK17}. The aim of
this paper is to analyze some of these observations and to propose an algebraic
stabilization for which optimal convergence rates can be observed in a wide
range of situations, in particular, for various types of meshes.

\section{Examples of algebraic stabilizations}
\label{s4}

In this section we present three examples of algebraic stabilizations based on
the papers \cite{Kuzmin07}, \cite{BJK17}, and \cite{JK21}, respectively. All 
these stabilizations fit into the framework of the previous section.

\subsection{Algebraic flux correction with the Kuzmin limiter}
\label{s41}

To derive an algebraic flux correction (AFC) scheme for the problem \eqref{21},
\eqref{21b}, one first introduces the artificial diffusion matrix 
$\mathbb D=(d_{ij})_{i,j=1}^N$ by
\begin{equation*}
   d_{ij}=d_{ji}=-\max\{a_{ij},0,a_{ji}\}\qquad\forall\,\,i\neq j\,,\qquad\qquad
   d_{ii}=-\sum_{j\neq i}\,d_{ij}\,.
\end{equation*}
Note that this matrix possesses the properties \eqref{eq_b1}--\eqref{eq_b4}.
If $({\mathbb D}\,{\rm U})_i$ is added to the left-hand side of \eqref{21},
one obtains a problem satisfying the DMP. However, this stabilized problem is
too diffusive. Therefore, one first adds the term $({\mathbb D}\,{\rm U})_i$ to 
both sides of \eqref{21}, uses the identity
\begin{displaymath}
   ({\mathbb D}\,{\rm U})_i=\sum_{j=1}^N\,f_{ij}\qquad\mbox{with}\qquad
   f_{ij}=d_{ij}\,(u_j-u_i)
\end{displaymath}
and then, on the right-hand side, one limits those anti-diffusive fluxes 
$f_{ij}$ that would otherwise cause spurious oscillations. The limiting is 
achieved by multiplying the fluxes by solution dependent limiters 
$\alpha_{ij}\in[0,1]$ satisfying
\begin{equation}
   \alpha_{ij}=\alpha_{ji}\,,\qquad i,j=1,\dots,N\,.\label{31}
\end{equation}
This leads to the algebraic problem \eqref{23}, \eqref{23b} with
\begin{equation}\label{afc-bij}
   b_{ij}({\rm U})=(1-\alpha_{ij}({\rm U}))\,d_{ij}\qquad\forall\,\,i\neq j\,,
   \qquad\quad
   b_{ii}({\rm U})=-\sum_{j\neq i}\,b_{ij}({\rm U})\,.
\end{equation}
This matrix $(b_{ij}({\rm U}))_{i,j=1}^N$ satisfies the assumptions
\eqref{eq_b1}--\eqref{eq_b4}. A theoretical analysis of this AFC scheme
concerning the solvability, local DMP and error estimation can be found in 
\cite{BJK16} where also a detailed derivation of the scheme is presented.

The properties of the above-described AFC scheme significantly depend on the 
choice of the limiters $\alpha_{ij}$. Here we present the Kuzmin limiter 
proposed in \cite{Kuzmin07} which was thoroughly investigated in \cite{BJK16}
and can be considered as a standard limiter for algebraic stabilizations of 
steady-state convection--diffusion--reaction equations.

To define the limiter of \cite{Kuzmin07}, one first computes, for
$i=1,\dots,M$,
\begin{equation}\label{46}
   P_i^+=\sum_{\mbox{\parbox{8mm}{\centerline{$\scriptstyle j\in S_i$}
   \vspace*{-1mm}
   \centerline{$\scriptstyle a_{ji}\le a_{ij}$}}}}\,f_{ij}^+\,,\quad\,\,
   P_i^-=\sum_{\mbox{\parbox{8mm}{\centerline{$\scriptstyle j\in S_i$}
   \vspace*{-1mm}
   \centerline{$\scriptstyle a_{ji}\le a_{ij}$}}}}\,f_{ij}^-\,,\quad\,\,
   Q_i^+=-\sum_{j\in S_i}\,f_{ij}^-\,,\quad\,\,
   Q_i^-=-\sum_{j\in S_i}\,f_{ij}^+\,,
\end{equation}
where $f_{ij}=d_{ij}\,(u_j-u_i)$,
$f_{ij}^+=\max\{0,f_{ij}\}$, and $f_{ij}^-=\min\{0,f_{ij}\}$. Then, one defines
\begin{equation}\label{R-definition}
   R_i^+=\min\left\{1,\frac{Q_i^+}{P_i^+}\right\},\quad
   R_i^-=\min\left\{1,\frac{Q_i^-}{P_i^-}\right\},\qquad
   i=1,\dots,M\,.
\end{equation}
If $P_i^+$ or $P_i^-$ vanishes, one sets $R_i^+=1$ or $R_i^-=1$, respectively.
At Dirichlet nodes, these quantities are also set to be $1$, i.e.,
\begin{equation}\label{eq:R-Dir}
   R_i^+=1\,,\quad R_i^-=1\,,\qquad i=M+1,\dots,N\,.
\end{equation}
Furthermore, one sets
\begin{equation}\label{alpha-definition}
        \widetilde\alpha_{ij}=\left\{
        \begin{array}{cl}
                R_i^+\quad&\mbox{if}\,\,\,f_{ij}>0\,,\\
                1\quad&\mbox{if}\,\,\,f_{ij}=0\,,\\
                R_i^-\quad&\mbox{if}\,\,\,f_{ij}<0\,,
        \end{array}\right.\qquad\qquad
        i,j=1,\dots,N\,.
\end{equation}
Finally, one defines
\begin{equation}\label{symm_alpha}
   \alpha_{ij}=\alpha_{ji}=\widetilde\alpha_{ij}\qquad\mbox{if}\quad
   a_{ji}\le a_{ij}\,,\qquad i,j=1,\dots,N\,.
\end{equation}

It was proved in \cite{BJK16} that $\alpha_{ij}({\rm U})(u_j-u_i)$ are
continuous functions of ${\rm U}\in{\mathbb R}^N$ so that the assumption (A1)
is satisfied for $b_{ij}({\rm U})$ defined by \eqref{afc-bij} with the Kuzmin 
limiter. The validity of (A2) was proved in \cite{Kno17} under the assumption
\begin{equation}\label{assumption_min}
   \min\{a_{ij},a_{ji}\}\le0\qquad
   \forall\,\,i=1,\dots,M\,,\,\,j=1,\dots,N\,,\,\,i\neq j\,.
\end{equation}
On the other hand, it was shown in \cite{Kno17} that the DMP generally does not 
hold if the condition \eqref{assumption_min} is not satisfied. Since the
convection matrix is skew-symmetric, the condition \eqref{assumption_min} can
be violated if the diffusion matrix has large positive entries (which may 
occur if the angles between facets of $\TT_h$ exceed $\pi/2$) or if the
reaction coefficient $c$ is large. As a remedy for the latter case, a lumping
of the reaction term was considered in \cite{BJK16}. This, however, may
increase the smearing of layers as demonstrated in \cite{JK21}. Let us mention
that, in the two-dimensional case and for $c=0$ or a lumped reaction term,
the validity of \eqref{assumption_min} is guaranteed for Delaunay meshes (i.e.,
meshes where the sum of any pair of angles opposite a common edge is smaller 
than, or equal to, $\pi$).

Since it is desirable that the DMP holds on arbitrary meshes and without a
lumping of the reaction term, it is necessary to apply other limiters or
different algebraic stabilizations. This will be the subject of the following
two sections.

\subsection{Algebraic flux correction with the BJK limiter}
\label{s42}

In this section, we again consider an AFC scheme, i.e., the matrix 
${\mathbb B}({\rm U})$ in \eqref{23}, \eqref{23b} is defined by 
\eqref{afc-bij}. A small difference to the previous section is that the matrix
$(a_{ij})_{i,j=1}^N$ is modified by
\begin{equation*}
   a_{ji}:=0\quad\mbox{if}\quad a_{ij}<0\,,\qquad 
   i=1,\dots,M\,,\,j=M+1,\dots,N\,.
\end{equation*}
This modification affects only the definition of the matrix $\mathbb D$
and reduces the amount of artificial diffusion introduced by the algebraic
stabilization. We shall describe the so-called BJK limiter proposed in
\cite{BJK17} using some ideas of \cite{Kuzmin12}. The definition of this 
limiter is inspired by the Zalesak algorithm \cite{Zalesak79} for the 
time-dependent case. 

The definition of the limiter again relies on local quantities 
$P_i^+$, $P_i^-$, $Q_i^+$, $Q_i^-$ which are now computed for $i=1,\dots,M$ by
\begin{align}
   P_i^+&=\sum_{j\in S_i}\,f_{ij}^+\,,\qquad
   P_i^-=\sum_{j\in S_i}\,f_{ij}^-\,,\label{eq:BJK_p}\\
   Q_i^+&=q_i\,(u_i-u_i^{\rm max})\,,\qquad
   Q_i^-=q_i\,(u_i-u_i^{\rm min})\,,\label{eq:BJK_q}
\end{align}
where again $f_{ij}=d_{ij}\,(u_j-u_i)$ and
\begin{equation*}
   u_i^{\max}= \max_{j\in S_i\cup\{i\}}\,u_j\,,\qquad
   u_i^{\min}= \min_{j\in S_i\cup\{i\}}\,u_j\,,\qquad
   q_i=\sum_{j\in S_i}\,d_{ij}\,.
\end{equation*}
Then, one defines
\begin{equation}\label{R-definition-BJK}
   R_i^+=\min\left\{1,\frac{\mu_i\,Q_i^+}{P_i^+}\right\},\quad
   R_i^-=\min\left\{1,\frac{\mu_i\,Q_i^-}{P_i^-}\right\},\qquad
   i=1,\dots,M\,,
\end{equation}
with fixed constants $\mu_i>0$. If $P_i^+$ or $P_i^-$ vanishes, one again 
sets $R_i^+=1$ or $R_i^-=1$, respectively. The definition \eqref{eq:R-Dir} of 
$R_i^\pm$ at Dirichlet nodes is applied, too, and one again defines the factors 
$\widetilde\alpha_{ij}$ by \eqref{alpha-definition}. Finally, the limiter
functions are defined by
\begin{equation}\label{eq:BJK_symm_alpha}
   \alpha_{ij}=\min\{\widetilde\alpha_{ij},\widetilde\alpha_{ji}\}\,,
   \qquad i,j=1,\dots,N\,.
\end{equation}

The validity of the assumptions (A1) and (A2) was proved in \cite{BJK17}
without any additional assumptions on the matrix $(a_{ij})_{i,j=1}^N$. Thus, in
particular, the DMP holds for arbitrary simplicial meshes and any nonnegative
reaction coefficient $c$. Moreover, it was shown in \cite{BJK17} that the
constants $\mu_i$ can be defined in such a way that the AFC scheme with the 
BJK limiter is linearity preserving, i.e., ${\mathbb B}(u)=0$ for 
$u\in P_1(\mathbb{R}^d)$. This property may lead to improved convergence
results, see, e.g., \cite{BBK17b,BJKR18}. 

To formulate a sufficient condition for the linearity preservation, we
introduce the patches
\begin{equation}\label{eq:patch}
   \Delta_i = \cup\{ T\in\TT_h\,:\,\,x_i\in T\}\,,\qquad i=1,\dots,M\,,
\end{equation}
consisting of simplices from $\TT_h$ sharing the vertex $x_i$. Then the
AFC scheme with the BJK limiter is linearity preserving if
\begin{equation}\label{eq:mu_i}
   \mu_i\ge\frac{\displaystyle\max_{x_j\in\partial\Delta_i}\,\vert x_i-x_j\vert}
                {\mbox{\rm dist}(x_i,\partial\Delta_i^{\rm conv})}\,,
   \qquad\quad i=1,\dots,M\,,
\end{equation}
where $\Delta_i^{\rm conv}$ is the convex hull of $\Delta_i$. It was also
proved in \cite{BJK17} that it suffices to set $\mu_i=1$ if the patch
$\Delta_i$ is symmetric with respect to the vertex $x_i$. Note that large 
values of the constants $\mu_i$ cause that more limiters $\alpha_{ij}$ are
equal to 1 and hence less artificial diffusion is added, which makes it 
possible to obtain sharp approximations of layers. On the other hand, however, 
large values of $\mu_i$'s also cause that the numerical solution of the 
nonlinear algebraic problem becomes more involved.

\subsection{Monotone upwind-type algebraically stabilized method}
\label{s43}

Although the BJK limiter presented in the previous section has nice
theoretical properties, numerical experiments revealed that it has also some
drawbacks in comparison with the Kuzmin limiter. In particular, the nonlinear
algebraic problems are much more difficult to solve and the approximate
solutions are sometimes less accurate away from layers. Therefore, another
approach based on the Kuzmin limiter was developed in \cite{Kno21,JK21}
that will be presented in this section.

As we mentioned in Section~\ref{s41}, the DMP generally does not hold for the 
AFC scheme with the Kuzmin limiter if the condition \eqref{assumption_min} is 
not satisfied. The need of \eqref{assumption_min} for proving the assumption 
(A2) is a consequence of the condition $a_{ji}\le a_{ij}$ used in 
\eqref{symm_alpha} to symmetrize the factors $\widetilde\alpha_{ij}$. 
A possible remedy is to replace \eqref{symm_alpha} by \eqref{eq:BJK_symm_alpha} 
and to define $P_i^\pm$ by \eqref{eq:BJK_p}. Then the DMP is satisfied without 
any additional condition on the matrix $(a_{ij})_{i,j=1}^N$ but the method is 
more diffusive then the scheme from Section~\ref{s41}, see \cite{Kno17}.

The inequality $a_{ji}<a_{ij}$ often means that the vertex $x_i$ lies in the 
upwind direction with respect to the vertex $x_j$, see \cite{Kno17} for a
discussion on this topic. Consequently, the use of the inequality 
$a_{ji}\le a_{ij}$ in \eqref{symm_alpha} causes that $\alpha_{ij}=\alpha_{ji}$ 
is defined using quantities computed at the upwind vertex of the edge with end 
points $x_i$, $x_j$.  It turns out that this feature has a positive influence 
on the quality of the approximate solutions and on the convergence of the 
iterative process for solving the nonlinear problem \eqref{23}, \eqref{23b}. 

In order to obtain a method possessing the mentioned upwind feature and
satisfying the DMP on arbitrary meshes, the definition of the matrix 
${\mathbb B}({\rm U})$ was changed in \cite{JK21} to
\begin{align}
   b_{ij}({\rm U})
   &=-\max\{\beta_{ij}({\rm U})\,a_{ij},0,\beta_{ji}({\rm U})\,a_{ji}\}\,,\qquad
   i,j=1,\dots,N\,,\,\,i\neq j\,,\label{asm-bij}\\
   b_{ii}({\rm U})
   &=-\sum_{j\neq i}\,b_{ij}({\rm U})\,,\qquad i=1,\dots,N\,,\label{asm-bii}
\end{align}
with some solution-dependent factors $\beta_{ij}({\rm U})\in[0,1]$. This matrix 
again satisfies the assumptions \eqref{eq_b1}--\eqref{eq_b4} but, in contrast 
to \eqref{afc-bij}, the formula \eqref{asm-bij} leads to a symmetric matrix 
${\mathbb B}({\rm U})$ also if the factors $\beta_{ij}$ are not symmetric. 
This makes it possible to get rid of the symmetry condition \eqref{31}.

If the condition \eqref{assumption_min} is satisfied, then
\begin{equation*}
        b_{ij}({\rm U})=\left\{
        \begin{array}{cl}
         \beta_{ij}({\rm U})\,d_{ij}\quad&\mbox{if}\,\,\,a_{ji}\le a_{ij}\,,\\
         \beta_{ji}({\rm U})\,d_{ij}\quad&\mbox{otherwise}\,,
        \end{array}\right.
\end{equation*}
for $i=1,\dots,M$ and $j=1,\dots,N$ with $i\neq j$. Thus, in this case, the 
definition \eqref{asm-bij} implicitly comprises the favourable upwind feature 
discussed above and the method \eqref{23}, \eqref{23b} can be again written in 
the form of an AFC scheme. Moreover, if the functions $\beta_{ij}$ form a 
symmetric matrix and $\alpha_{ij}=1-\beta_{ij}$, then the definitions 
\eqref{afc-bij} and \eqref{asm-bij}, \eqref{asm-bii} are equivalent.

Thus, let us consider the algebraic problem \eqref{23}, \eqref{23b} with the
artificial diffusion matrix given by \eqref{asm-bij} and \eqref{asm-bii} and
with any functions $\beta_{ij}:{\mathbb R}^N\to[0,1]$ satisfying, for any $i,j\in\{1,\dots,N\}$,
\begin{equation}
   \mbox{if $a_{ij}>0$, then $\beta_{ij}({\rm U})(u_j-u_i)$ is a continuous 
          function of ${\rm U}\in{\mathbb R}^N$}\,.\label{asm_beta2}
\end{equation}
Then one has the following existence result.

\begin{theorem}\label{asm_existence}
Let the matrix $(b_{ij}({\rm U}))_{i,j=1}^N$ be defined by \eqref{asm-bij} and 
\eqref{asm-bii} with functions $\beta_{ij}:{\mathbb R}^N\to[0,1]$ satisfying 
\eqref{asm_beta2} for any $i,j\in\{1,\dots,N\}$. Then Assumption~(A1) is 
satisfied and the nonlinear algebraic problem \eqref{23}, \eqref{23b} has a 
solution.
\end{theorem}

\begin{proof} See \cite{JK21}.\end{proof}

Rewriting the definition of the Kuzmin limiter under the condition
\eqref{assumption_min}, the following definition of $\beta_{ij}$ was introduced
in \cite{JK21}. First, for any $i\in\{1,\dots,M\}$, one computes
\begin{alignat}{2}
   &P_i^+=\sum_{\mbox{\parbox{8mm}{\centerline{$\scriptstyle j\in S_i$}
   \vspace*{-1mm}
   \centerline{$\scriptstyle a_{ij}>0$}}}}\,a_{ij}\,(u_i-u_j)^+\,,\qquad\quad
   &&P_i^-=\sum_{\mbox{\parbox{8mm}{\centerline{$\scriptstyle j\in S_i$}
   \vspace*{-1mm}
   \centerline{$\scriptstyle a_{ij}>0$}}}}\,a_{ij}\,(u_i-u_j)^-\,,\label{def-beta_ij_1}
   \\[1mm]
   &Q_i^+=\sum_{j\in S_i}\,s_{ij}\,(u_j-u_i)^+\,,\qquad\quad
   &&Q_i^-=\sum_{j\in S_i}\,s_{ij}\,(u_j-u_i)^-\,,\label{def-beta_ij_2}
\end{alignat}
with
\begin{equation*}
   s_{ij}=\max\{\vert a_{ij}\vert,a_{ji}\}\,.
\end{equation*}
Then, one defines $R_i^\pm$ by \eqref{R-definition} and \eqref{eq:R-Dir}, and
sets
\begin{equation}\label{eq:betaij}
        \beta_{ij}=\left\{
        \begin{array}{ll}
                1-R_i^+\quad&\mbox{if}\,\,\,u_i>u_j\,,\\
                0\quad&\mbox{if}\,\,\,u_i=u_j\,,\\
                1-R_i^-\quad&\mbox{if}\,\,\,u_i<u_j\,,
        \end{array}\right.\qquad\qquad i,j=1,\dots,N\,.
\end{equation}

It was proved in \cite{JK21} that the resulting method satisfies the
assumptions (A1) and (A2) without any additional assumptions on the matrix
$(a_{ij})_{i,j=1}^N$. Thus, the DMP holds on arbitrary simplicial meshes and
for any nonnegative reaction coefficient $c$. Due to the above-discussed upwind
feature, the name Monotone Upwind-type Algebraically Stabilized (MUAS) method
was introduced in \cite{JK21}.

If the condition \eqref{assumption_min} holds, then the only difference between
the MUAS method and the AFC scheme with the Kuzmin limiter is the definition of 
$Q_i^\pm$ since the relations \eqref{46} give \eqref{def-beta_ij_2} with
$s_{ij}=\vert d_{ij}\vert$. In the convection-dominated regime, the difference 
is negligible and both methods lead to almost the same results. Therefore, the
MUAS method preserves the advantages of the AFC scheme from Section~\ref{s41}
which are available under the condition \eqref{assumption_min}. Note that, 
without the assumption \eqref{assumption_min}, the application of the AFC 
scheme with the Kuzmin limiter does not make much sense since the main goal of 
the AFC, i.e., the validity of the DMP, is not achieved in general. In the
diffusion-dominated case, the use of $s_{ij}$ instead of $\vert d_{ij}\vert$ 
may improve the accuracy and convergence behaviour when non-Delaunay meshes are
used, see \cite{JK21}.

\section{Numerical and analytical studies of AFC schemes}
\label{s5}

The convergence properties of the AFC scheme with the Kuzmin limiter from
Section~\ref{s41} were thoroughly tested in \cite{BJK16} for various grids and 
the following example.

\begin{example}\label{ex:smooth}
Problem \eqref{strong-steady} is considered with $\Omega = (0,1)^2$, with 
different values of $\varepsilon$, and with $\bb = (3,2)^T$, $c=1$, $u_b=0$, 
and the right-hand side $g$ chosen so that
$$
u(x,y) = 100\,x^2\,(1-x)^2\,y\,(1-y)\,(1-2y)
$$
is the solution of \eqref{strong-steady}.
\end{example}

The coarsest levels of the grids considered in \cite{BJK16} are shown in 
Fig.~\ref{fig:grids1}.
\begin{figure}[t]
\centerline{
\includegraphics[width=0.185\textwidth]{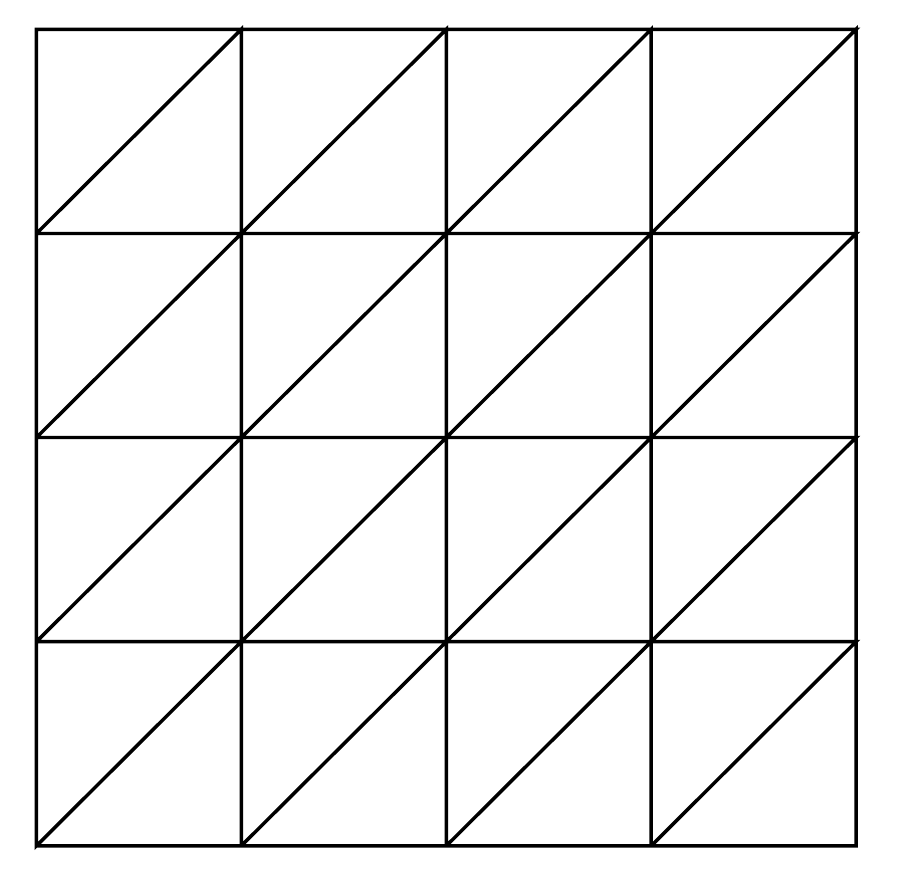}\hspace*{1ex}
\includegraphics[width=0.185\textwidth]{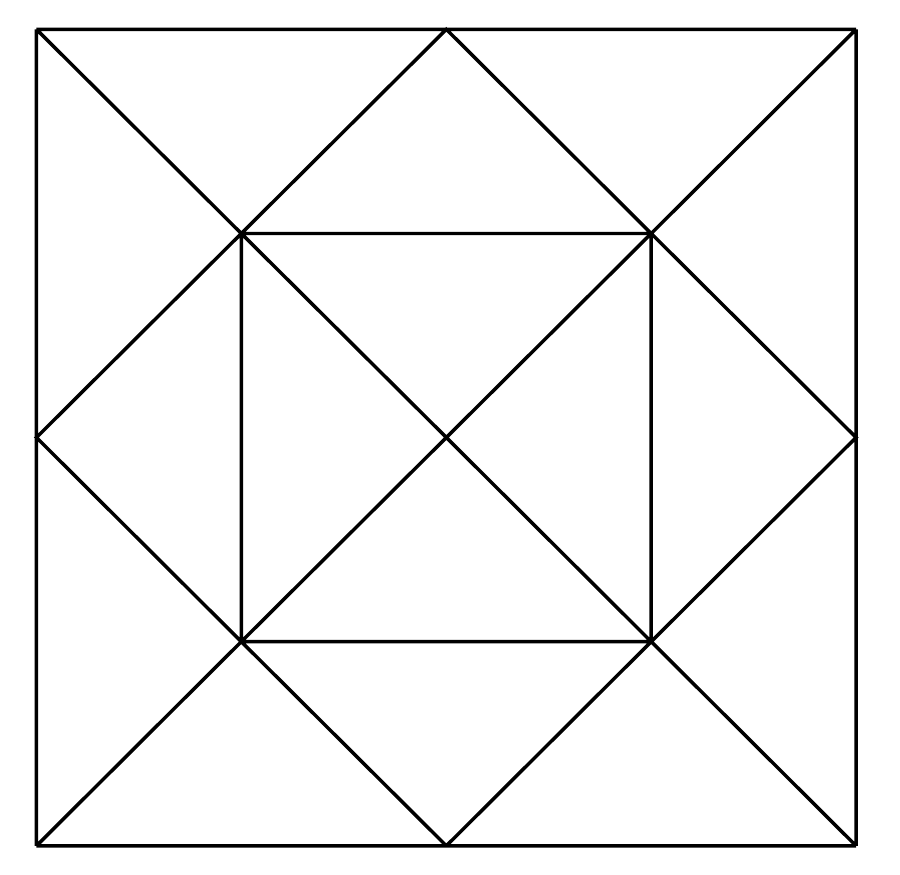}\hspace*{1ex}
\includegraphics[width=0.185\textwidth]{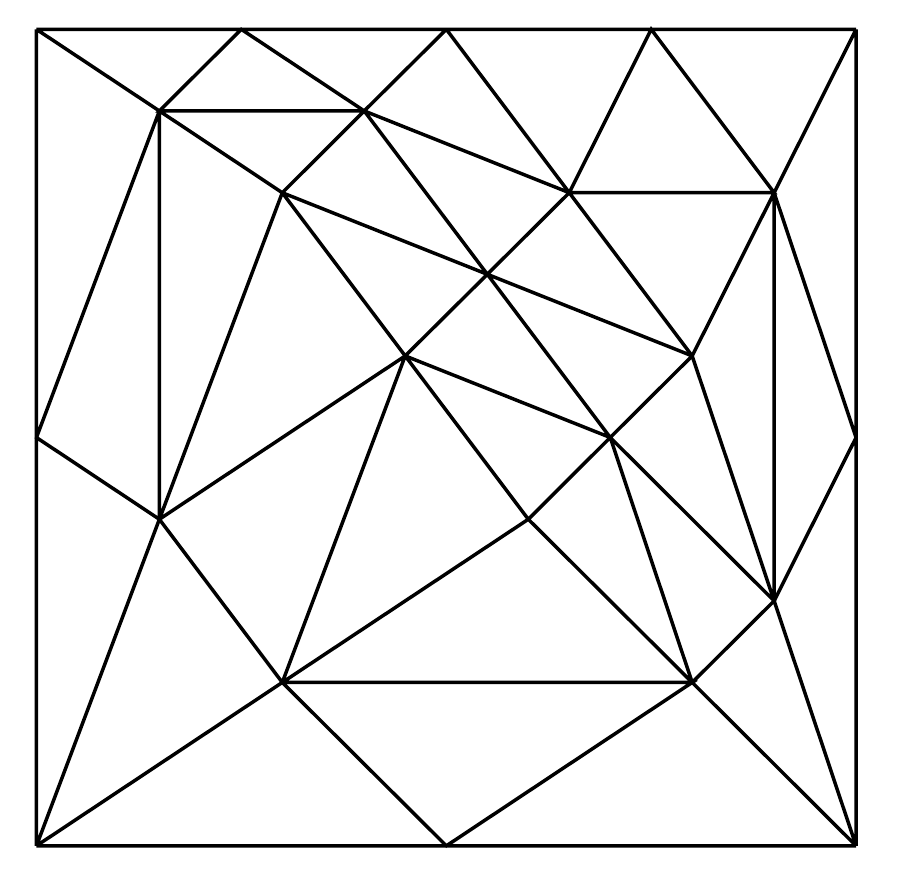}\hspace*{1ex}
\includegraphics[width=0.185\textwidth]{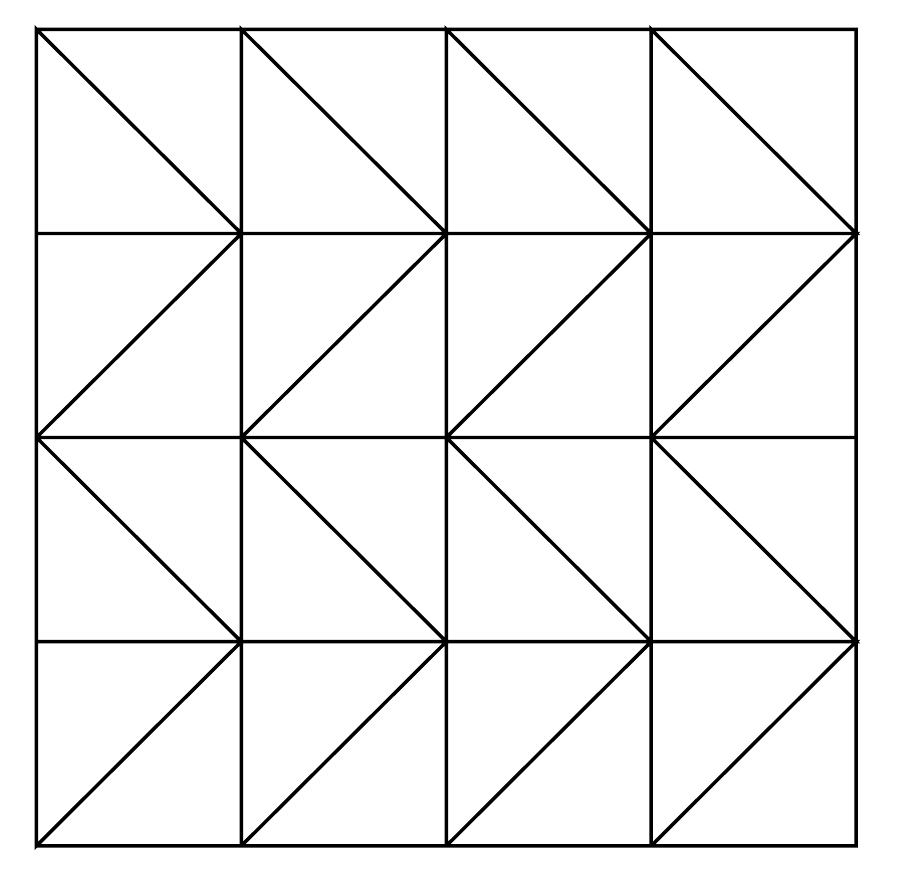}\hspace*{1ex}
\includegraphics[width=0.185\textwidth]{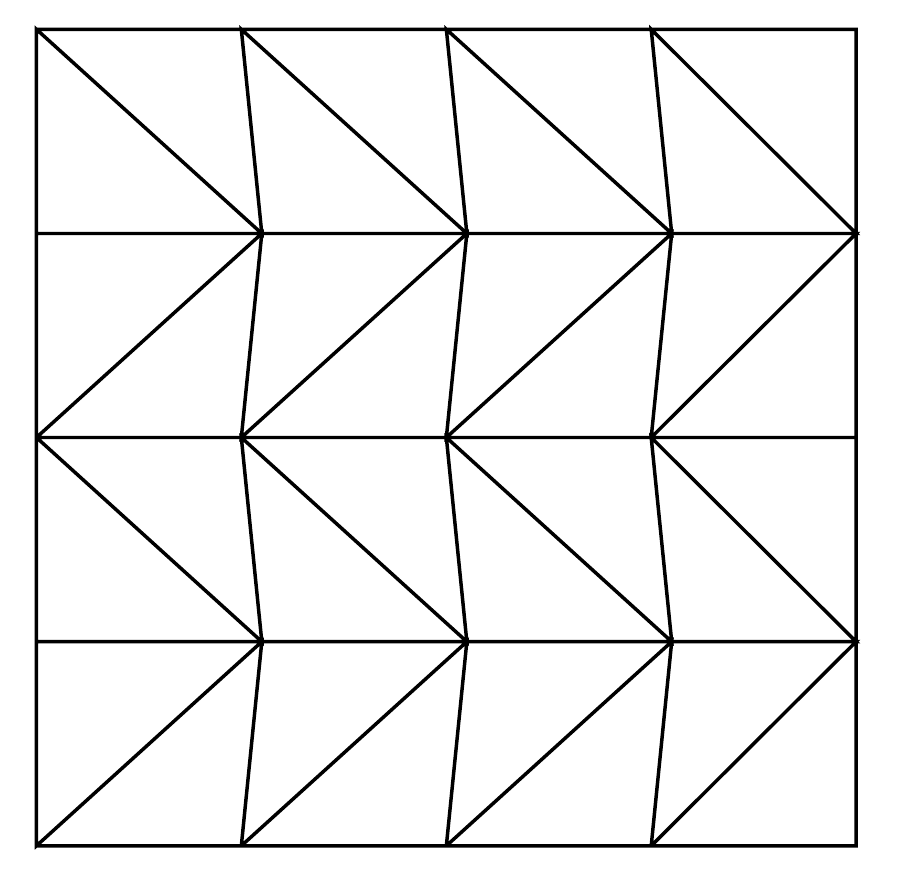}
}
\caption{Grids 1 -- 5 (left to right).}\label{fig:grids1}
\end{figure}
Grids 1, 2, and 3 were refined uniformly whereas Grid 4 was always obtained 
from Grid 1 by changing the directions of the diagonals in even rows of squares (from below). Grid 5 was obtained from Grid 4 by shifting interior nodes to the 
right by the tenth of the horizontal mesh width on each even horizontal mesh 
line. Note that Grids 3 and 5 are not of Delaunay type.

Errors of the approximate solutions of Example~\ref{ex:smooth} with 
$\varepsilon=10^{-8}$ computed using the AFC scheme with the
Kuzmin limiter for Grid 1 can be seen in Table~\ref{tab1}. The results slightly
differ from those in \cite{BJK16} since, in contrast to the present paper, a 
lumping of the reaction term was applied in \cite{BJK16}. The value of $ne$ 
represents the number of edges along one horizontal mesh line (thus, $ne=4$ 
for Grid~1 in Fig.~\ref{fig:grids1}). One observes the usual optimal orders of
convergence with respect to the $L^2$ norm and the $H^1$ seminorm. Moreover,
the convergence order with respect to the norm $\|\cdot\|_h^{}$ is much higher
than predicted by \eqref{eq:error_est}. However, if the computation is repeated
on Grid~4, one observes in Table~\ref{tab2} that the convergence orders with 
respect to all three norms deteriorate by 1 and, in particular, one has no 
error reduction with respect to the $H^1$ seminorm. A similar behaviour can be 
observed for Grids 3 and 5. For Grid 2, the deterioration of the convergence is 
less pronounced but the convergence orders are also far from being optimal (see 
\cite{BJK16} for the case with a lumped reaction term leading to similar 
results). Let us mention that, in all these computations, the matrix
$(a_{ij})_{i,j=1}^N$ satisfies the condition \eqref{assumption_min}, which 
guarantees the validity of the DMP.
\begin{table}[h]
\begin{center}
\begin{minipage}{274pt}
\caption{Example~\ref{ex:smooth}: $\varepsilon=10^{-8}$, numerical results for
Grid 1 computed using the AFC scheme with the Kuzmin limiter}
\label{tab1}
\begin{tabular}{@{}rcccccc@{}}
\toprule
$ne$ & $\|u-u_h\|_{0,\Omega}^{}$ & order &   $\vert u-u_h\vert_{1,\Omega}^{}$ &
order & $\|u-u_h\|_h^{}$ & order\\
\midrule
  16 &  1.934e$-$2 &  1.60 & 4.937e$-$1 &  0.98 & 5.007e$-$2 &  1.87\\ 
  32 &  5.359e$-$3 &  1.85 & 2.305e$-$1 &  1.10 & 1.149e$-$2 &  2.12\\
  64 &  1.385e$-$3 &  1.95 & 1.082e$-$1 &  1.09 & 2.649e$-$3 &  2.12\\
 128 &  3.442e$-$4 &  2.01 & 5.154e$-$2 &  1.07 & 6.152e$-$4 &  2.11\\
 256 &  8.536e$-$5 &  2.01 & 2.566e$-$2 &  1.01 & 1.586e$-$4 &  1.96\\
 512 &  2.126e$-$5 &  2.01 & 1.342e$-$2 &  0.93 & 3.876e$-$5 &  2.03\\
\botrule
\end{tabular}
\end{minipage}
\end{center}
\end{table}
\begin{table}[h]
\begin{center}
\begin{minipage}{274pt}
\caption{Example~\ref{ex:smooth}: $\varepsilon=10^{-8}$, numerical results for
Grid 4 computed using the AFC scheme with the Kuzmin limiter}
\label{tab2}
\begin{tabular}{@{}rcccccc@{}}
\toprule
$ne$ & $\|u-u_h\|_{0,\Omega}^{}$ & order &   $\vert u-u_h\vert_{1,\Omega}^{}$ &
order & $\|u-u_h\|_h^{}$ & order\\
\midrule
  16 &  2.019e$-$2 &  1.65 & 6.005e$-$1 &  0.68 & 5.663e$-$2 &  1.74\\
  32 &  6.285e$-$3 &  1.68 & 4.832e$-$1 &  0.31 & 2.138e$-$2 &  1.41\\
  64 &  2.308e$-$3 &  1.45 & 4.549e$-$1 &  0.09 & 9.485e$-$3 &  1.17\\
 128 &  1.092e$-$3 &  1.08 & 4.442e$-$1 &  0.03 & 4.490e$-$3 &  1.08\\
 256 &  5.543e$-$4 &  0.98 & 4.368e$-$1 &  0.02 & 2.187e$-$3 &  1.04\\
 512 &  2.823e$-$4 &  0.97 & 4.327e$-$1 &  0.01 & 1.083e$-$3 &  1.01\\
\botrule
\end{tabular}
\end{minipage}
\end{center}
\end{table}

On the other hand, there are various grids for which optimal convergence orders 
can be observed. Examples of such grids are given in Fig.~\ref{fig:grids2}. The 
finer variants of Grid~6 are obtained by uniform refinement like for Grid~1 
whereas Grid~7 is obtained from Grid~6 by changing the directions of some of 
the diagonals. Grid~8 is obtained from Grid~6 by adding the second diagonal in 
each small square. Finer variants of Grid~9 are also not constructed by 
refining the coarse level but each level is constructed separately, cf.~the 
rightmost grid in Fig.~\ref{fig:grids2}. Obviously, the basic difference 
between Grids 2--5 and Grids 1 and 6--9 is that, in the latter case, (most of) 
the patches $\Delta_i$ defined by \eqref{eq:patch} are symmetric. Thus, it 
seems that the local symmetry of the grids is important for optimal convergence 
rates.
\begin{figure}[b]
\centerline{
\includegraphics[width=0.185\textwidth]{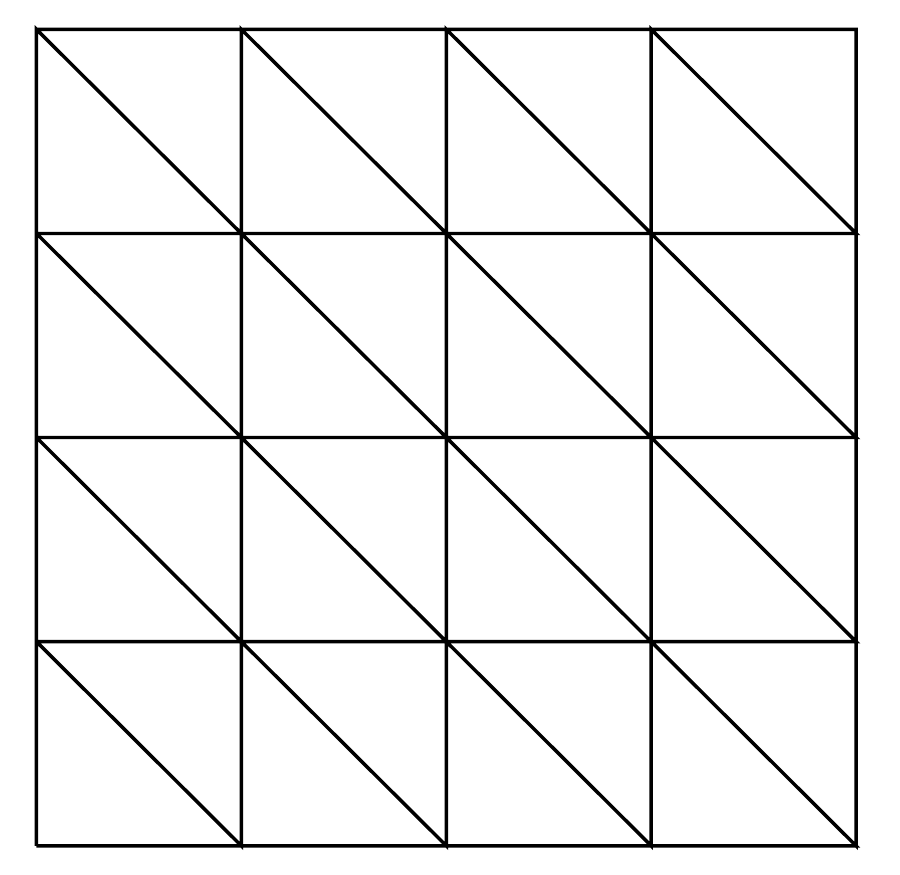}\hspace*{1ex}
\includegraphics[width=0.185\textwidth]{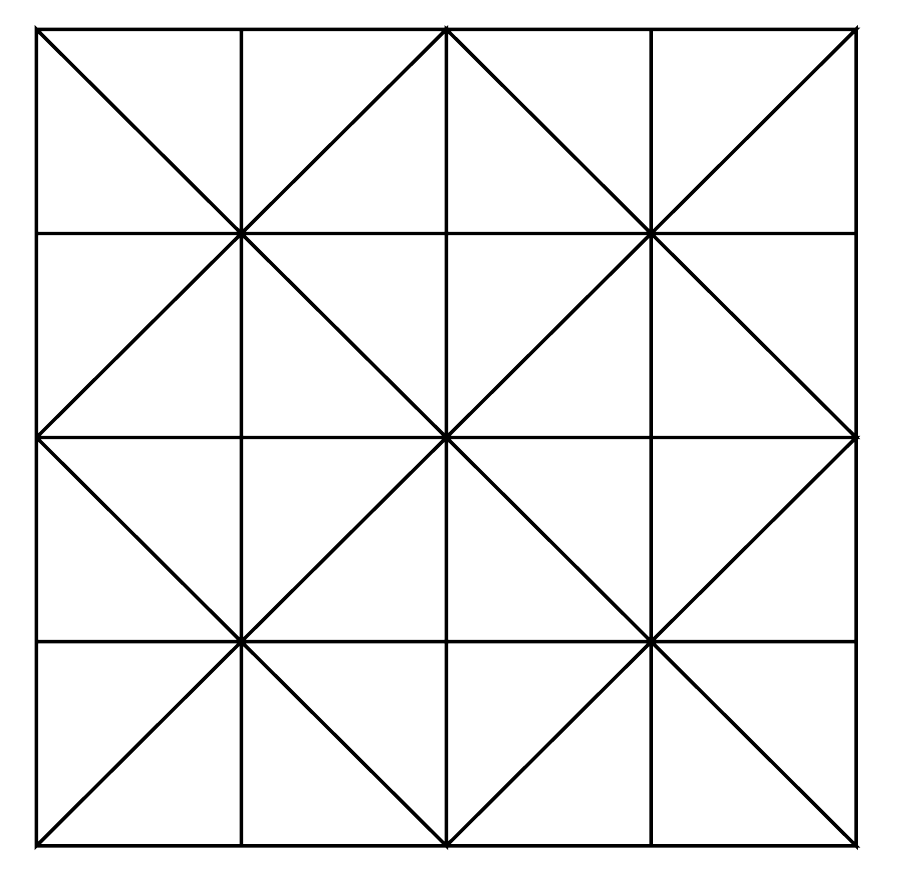}\hspace*{1ex}
\includegraphics[width=0.185\textwidth]{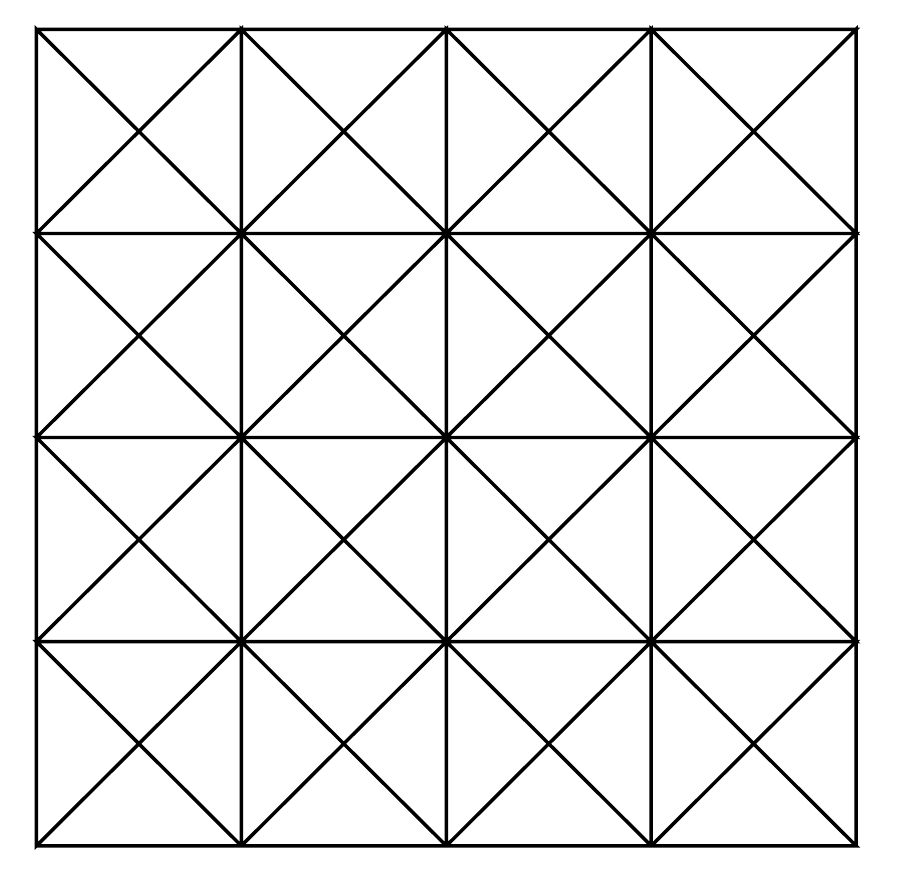}\hspace*{1ex}
\includegraphics[width=0.185\textwidth]{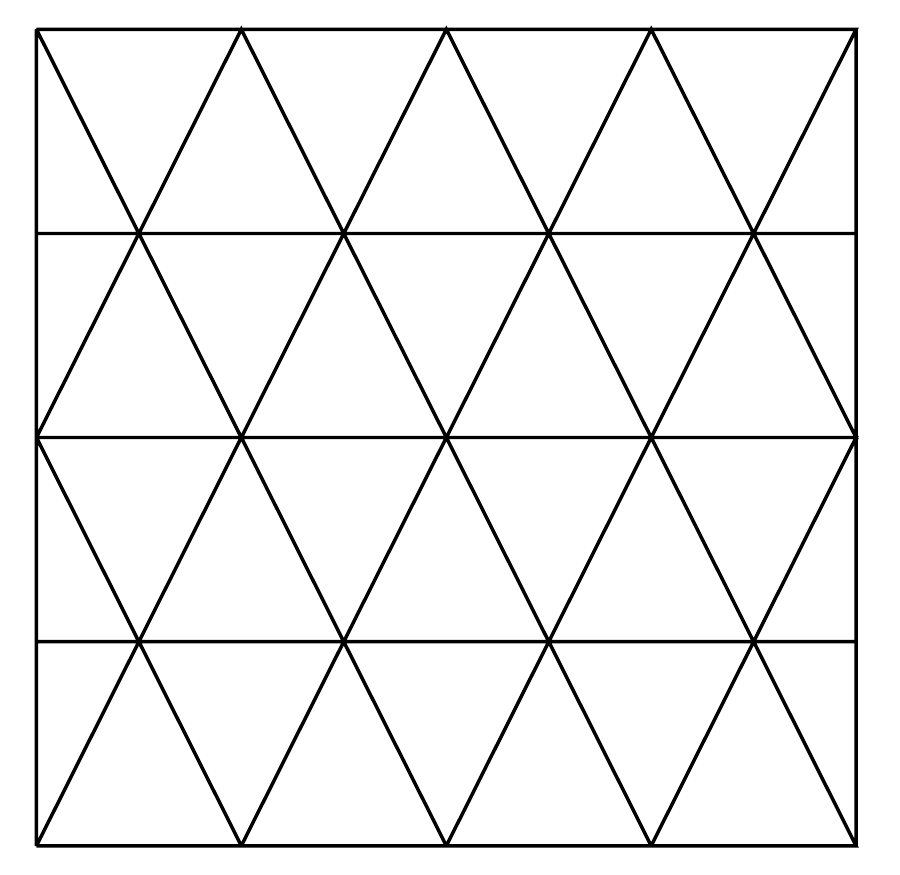}\hspace*{1ex}
\includegraphics[width=0.185\textwidth]{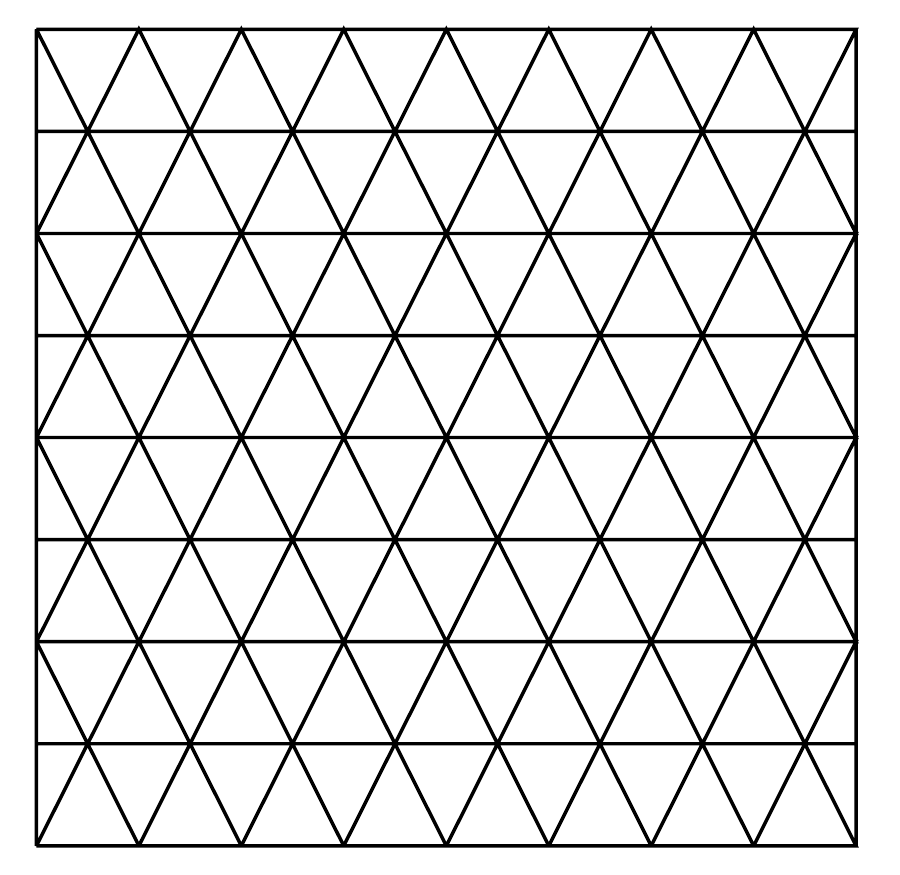}
}
\caption{Grids 6 -- 9 (left to right) and a finer variant of Grid 9 (rightmost).}
\label{fig:grids2}
\end{figure}

To understand why the approximate solutions on Grid~4 do not converge in the 
$H^1$ seminorm, let us have a look at the graphs of some of these solutions.
Fig.~\ref{fig:ex1} (left) shows the solution computed for $ne=32$ and it can 
be seen that the solution is polluted by an oscillating component (for the sake
of clarity, the solution is drawn only along grid lines of Grid~4 which
are parallel to the coordinate axes). This is also
clearly seen from Fig.~\ref{fig:ex1} (right) which shows the wildly oscillating
error $u_h-i_hu$, where $i_h$ is the usual Lagrange interpolation operator. 
The observed structure of the solution remains preserved also on finer meshes. 
Fig.~\ref{fig:ex1_cuts} shows the errors $u_h-i_hu$ along the line
$x=0.25$ on three successive meshes and indicates that the $H^1$ seminorm of 
the error will not change significantly when switching to finer meshes (notice 
the different scales on the vertical axes). It should be mentioned that this 
type of oscillations does not represent a violation of the DMP. The oscillatory
behaviour of the approximate solutions suggests that the accuracy might be
improved by a local averaging. This is indeed possible but the convergence
rates generally still remain suboptimal.

\begin{figure}[t]
\centerline{
\includegraphics[width=0.55\textwidth]{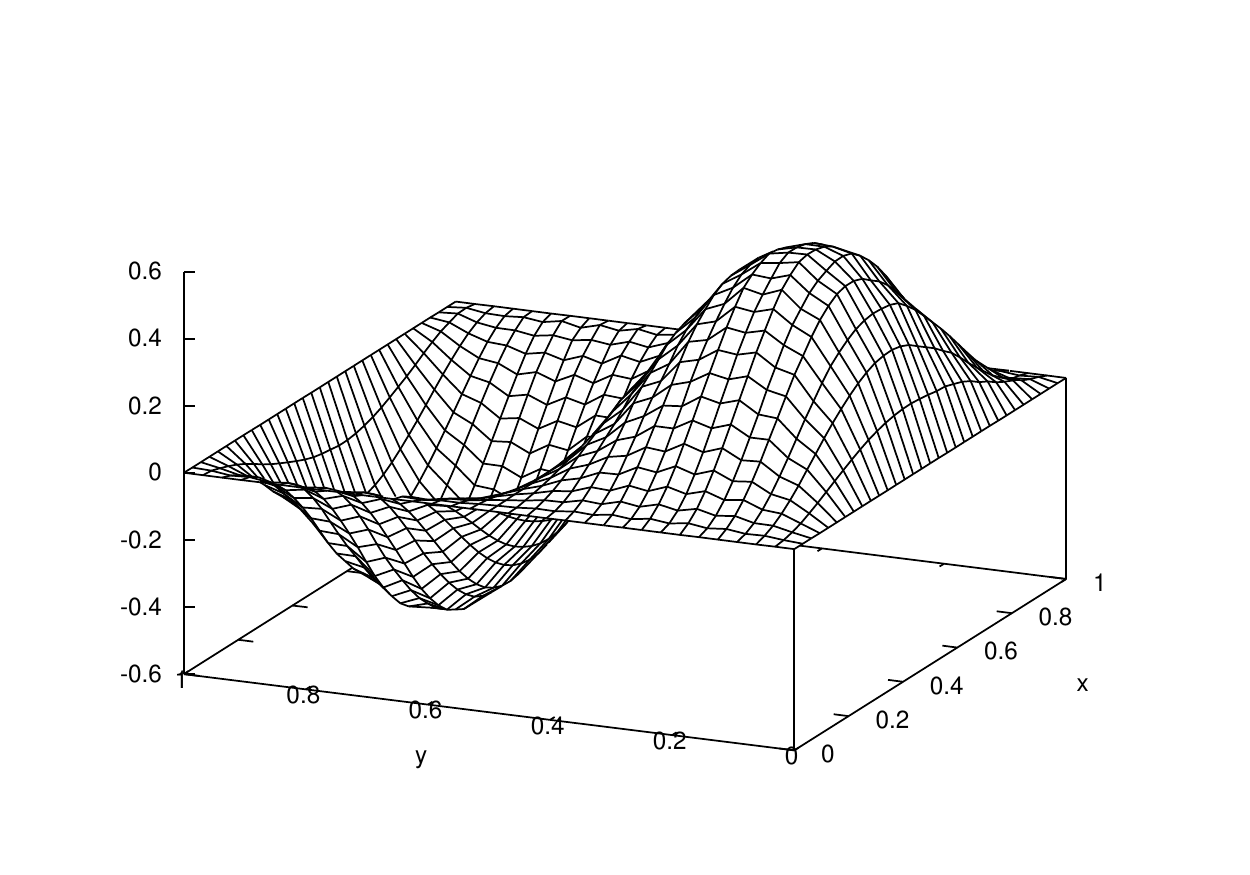}\hspace*{-2ex}
\includegraphics[width=0.55\textwidth]{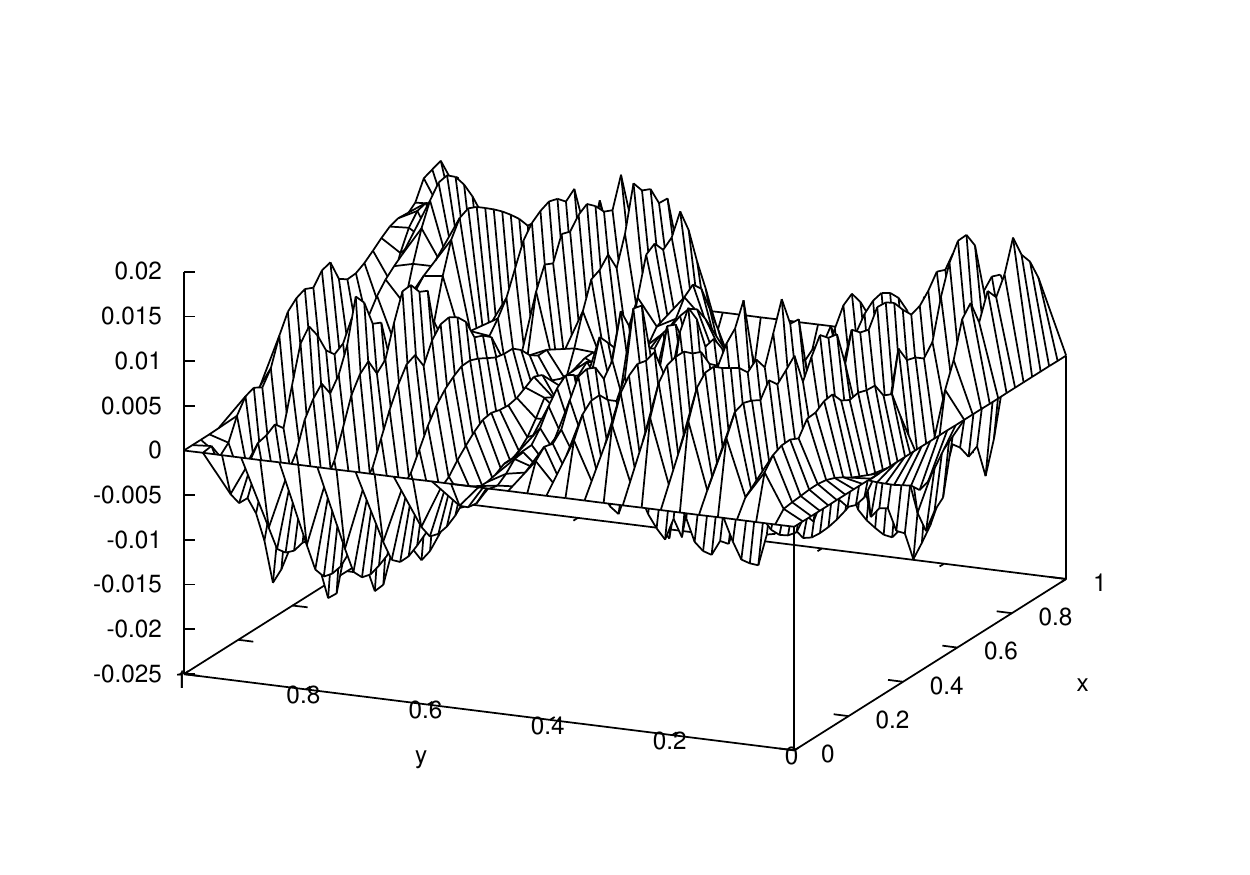}
}
\caption{Example~\ref{ex:smooth}: $\varepsilon=10^{-8}$, approximate solution 
computed using the AFC scheme with the Kuzmin limiter on Grid~4 with $ne=32$ 
(left) and the corresponding error $u_h-i_hu$ (right)}
\label{fig:ex1}
\end{figure}
\begin{figure}[t]
\centerline{
\includegraphics[width=0.32\textwidth]{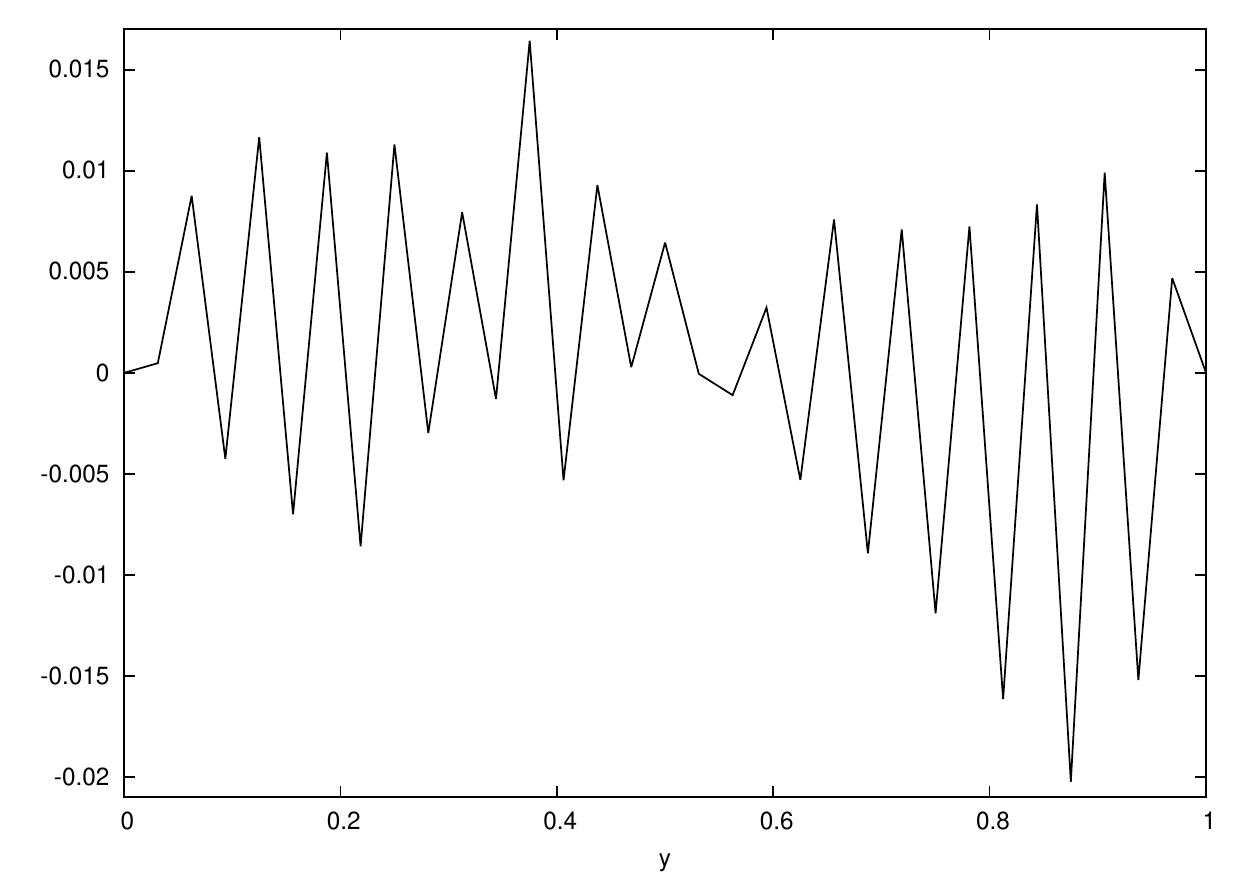}\hspace*{0.5ex}
\includegraphics[width=0.32\textwidth]{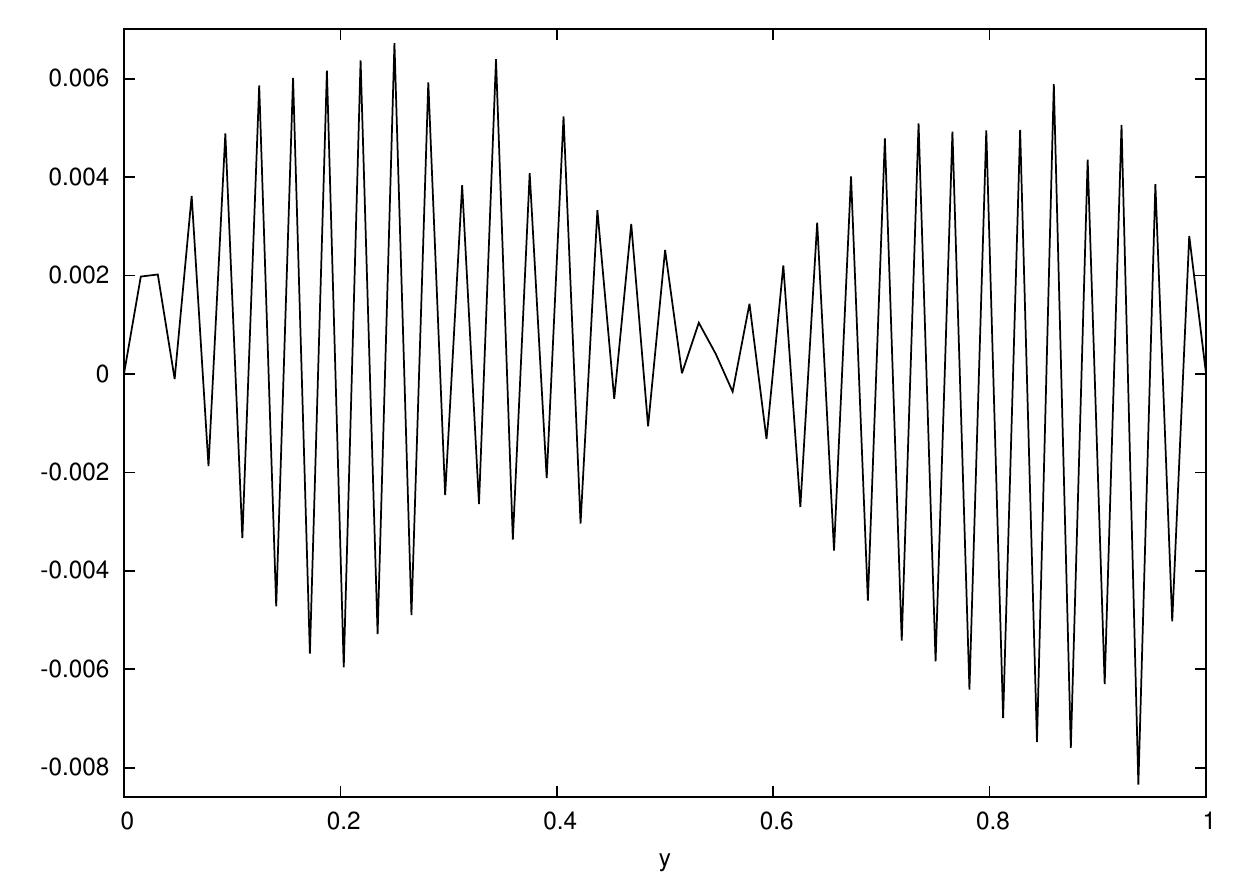}\hspace*{0.5ex}
\includegraphics[width=0.32\textwidth]{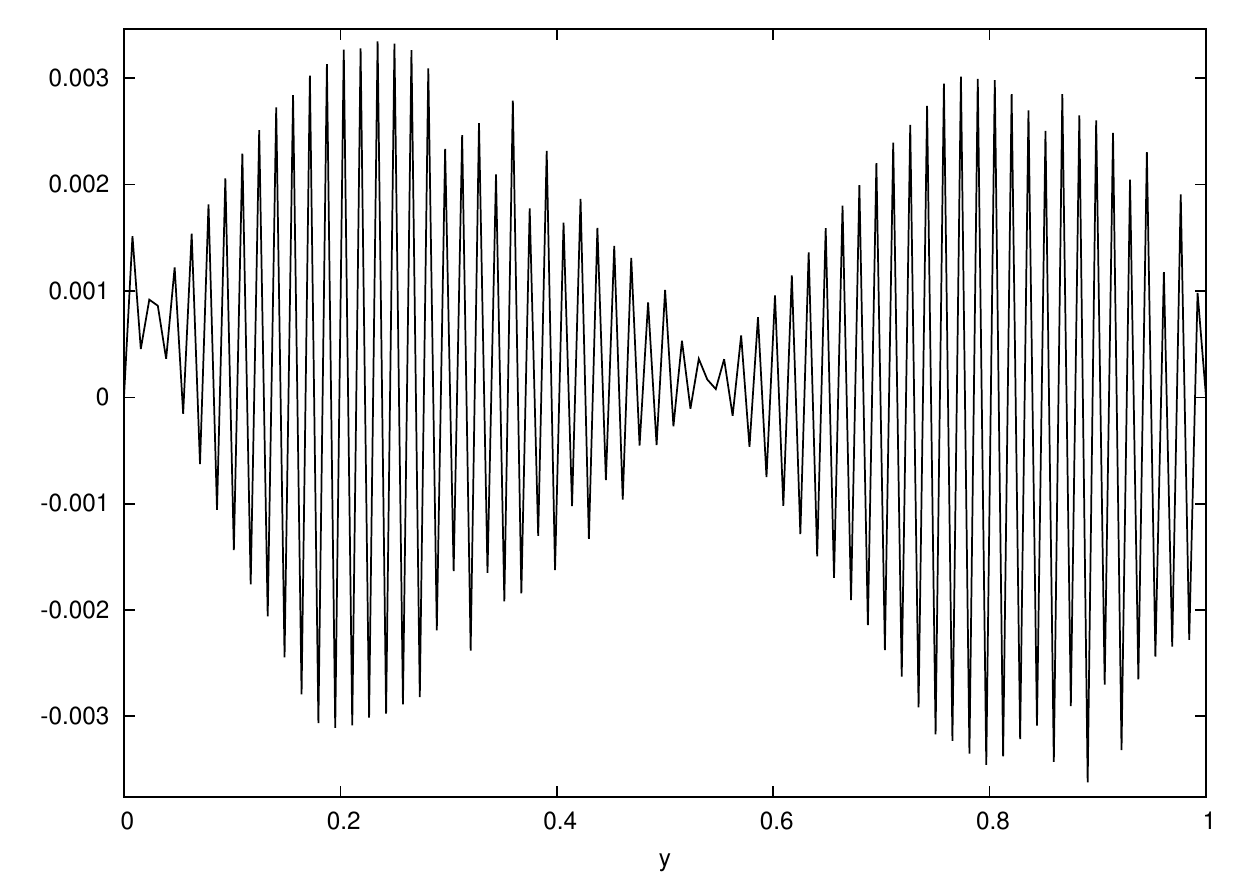}
}
\caption{Example~\ref{ex:smooth}: $\varepsilon=10^{-8}$, errors $u_h-i_hu$ 
along the line $x=0.25$ for approximate solutions computed using the AFC scheme 
with the Kuzmin limiter on Grid~4 with $ne=32$, $ne=64$ and $ne=128$ (left to 
right)}
\label{fig:ex1_cuts}
\end{figure}

Figs.~\ref{fig:ex1} and \ref{fig:ex1_cuts} explain why the $H^1$ seminorm of
the error of the the approximate solution does not tend to zero on Grid~4 and 
now the main question is why the observed oscillations are not suppressed by 
the algebraic stabilization. To answer this question, we shall consider simpler 
examples than Example~\ref{ex:smooth}. We start with the following almost 
trivial case.

\begin{example}\label{ex:linear}
Problem \eqref{strong-steady} is considered with $\Omega = (0,1)^2$, 
$\varepsilon=10^{-8}$, $\bb = (1,0)^T$, $c=0$, $g=1$, and $u_b(x,y)=x$.
\end{example}

Of course, the exact solution of this example is $u(x,y)=x$ and hence the
Galerkin FEM gives the exact solution on any mesh. However, if one applies the 
AFC scheme with the Kuzmin limiter on Grid~4, one obtains the oscillating
solution shown in Fig.~\ref{fig:ex23} (left). Again, the structure of the 
solution is preserved also on finer meshes. Moreover, numerical tests show that 
the size of the oscillations is proportional to $h$ so that one can again 
expect that the $H^1$ seminorm of the error will not tend to zero. This is 
confirmed by the results shown in Table~\ref{s3}.

\begin{figure}[t]
\centerline{
\includegraphics[width=0.55\textwidth]{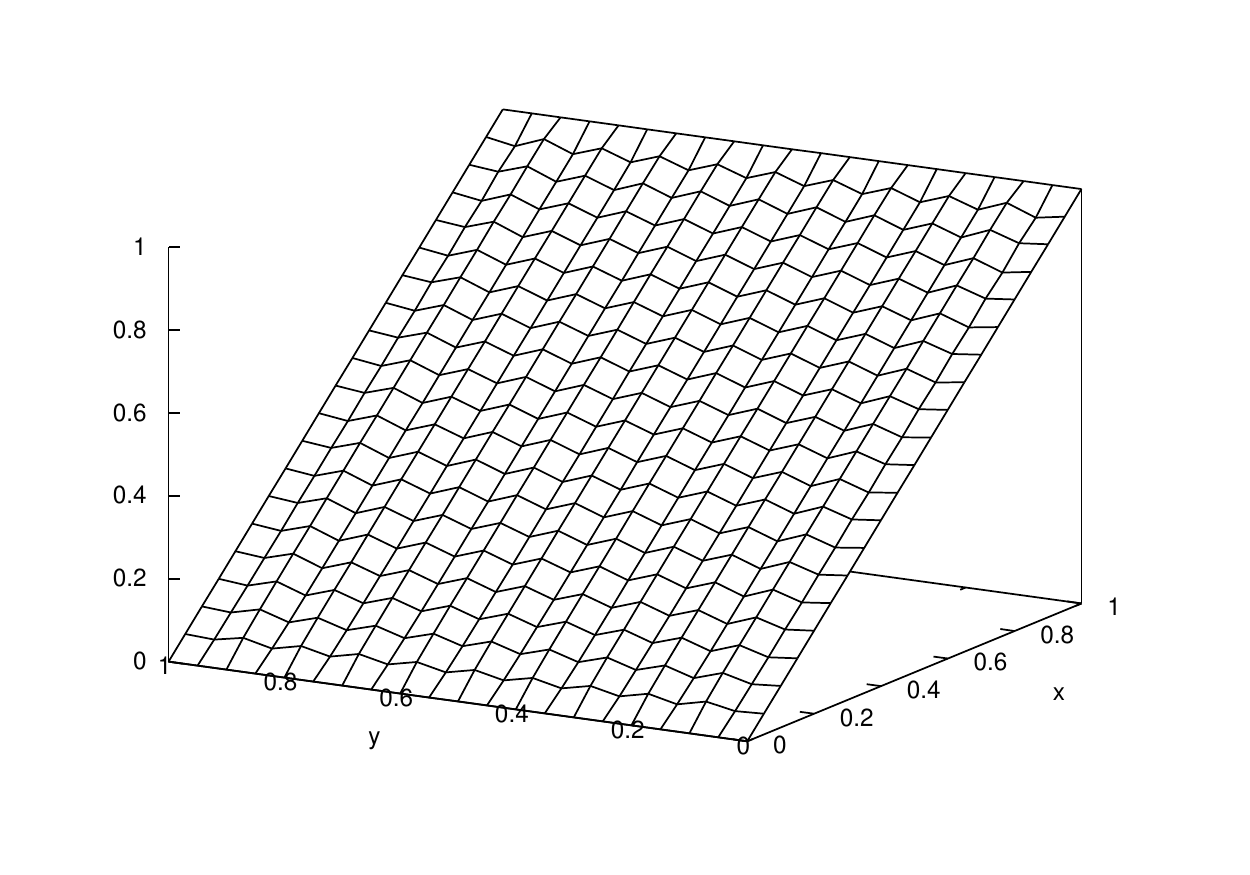}\hspace*{-2ex}
\includegraphics[width=0.55\textwidth]{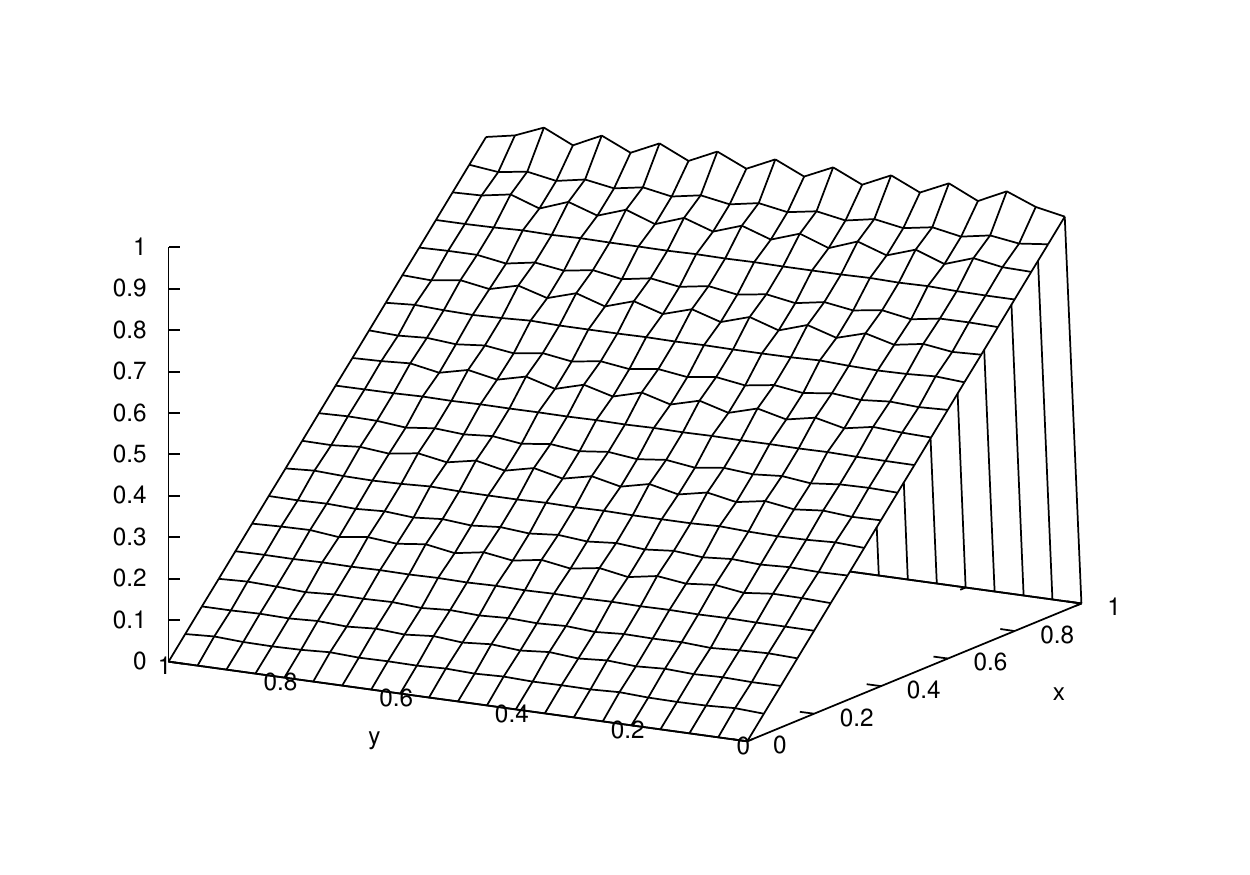}
}
\caption{Approximate solutions computed using the AFC scheme with the Kuzmin 
limiter on Grid~4 with $ne=20$: Example~\ref{ex:linear} (left) and 
Example~\ref{ex:exp_layer} computed with the modification \eqref{eq:mod_rb},
\eqref{eq:mu} (right) 
}
\label{fig:ex23}
\end{figure}

\begin{table}[b]
\begin{center}
\begin{minipage}{274pt}
\caption{Example~\ref{ex:linear}: numerical results for
Grid~4 computed using the AFC scheme with the Kuzmin limiter}
\label{tab3}
\begin{tabular}{@{}rcccccc@{}}
\toprule
$ne$ & $\|u-u_h\|_{0,\Omega}^{}$ & order &   $\vert u-u_h\vert_{1,\Omega}^{}$ &
order & $\|u-u_h\|_h^{}$ & order\\
\midrule
  16 &  8.104e$-$3 &  0.82 & 4.401e$-$1 & -0.20 & 1.179e$-$2 &  0.78\\
  32 &  4.291e$-$3 &  0.92 & 4.700e$-$1 & -0.09 & 6.227e$-$3 &  0.92\\
  64 &  2.204e$-$3 &  0.96 & 4.851e$-$1 & -0.05 & 3.157e$-$3 &  0.98\\
 128 &  1.117e$-$3 &  0.98 & 4.926e$-$1 & -0.02 & 1.580e$-$3 &  1.00\\
 256 &  5.618e$-$4 &  0.99 & 4.963e$-$1 & -0.01 & 7.893e$-$4 &  1.00\\
 512 &  2.817e$-$4 &  1.00 & 4.982e$-$1 & -0.01 & 3.974e$-$4 &  0.99\\
\botrule
\end{tabular}
\end{minipage}
\end{center}
\end{table}

Before we start our analytical investigations of this surprising observation, 
let us have a closer look at Grid~4 and the matrix entries corresponding to 
Example~\ref{ex:linear}. First, note that all the patches $\Delta_i$ defined in 
\eqref{eq:patch} have the same geometry for Grid~4 but they possess two types 
of orientation with respect to the constant convection vector $\bb$. This can 
be seen in Fig.~\ref{fig_grid4}
\begin{figure}[t]
\begin{center}
\setlength{\unitlength}{1cm}
\begin{picture}(2.6,3.6)(-0.3,-0.4)

\thicklines

\put(0,0){\line(1,0){2}} 
\put(0,1){\line(1,0){2}} 
\put(0,2){\line(1,0){2}} 
\put(0,3){\line(1,0){2}} 
\put(0,0){\line(0,1){3}}
\put(1,0){\line(0,1){3}}
\put(2,0){\line(0,1){3}}

\put(0,1){\line(1,-1){1}}
\put(1,1){\line(1,-1){1}}
\put(0,1){\line(1,1){1}}
\put(1,1){\line(1,1){1}}
\put(0,3){\line(1,-1){1}}
\put(1,3){\line(1,-1){1}}

\put(0,1){\circle*{0.1}}
\put(0,2){\circle*{0.1}}
\put(0,3){\circle*{0.1}}
\put(1,0){\circle*{0.1}}
\put(1,1){\circle*{0.1}}
\put(1,2){\circle*{0.1}}
\put(1,3){\circle*{0.1}}
\put(2,0){\circle*{0.1}}
\put(2,1){\circle*{0.1}}
\put(2,2){\circle*{0.1}}

\put(1.18,2.16){\makebox(0,0){$A$}}
\put(0.84,1.16){\makebox(0,0){$B$}}
\put(2.2,2){\makebox(0,0){$C$}}
\put(1,3.2){\makebox(0,0){$D$}}
\put(-0.21,3.05){\makebox(0,0){$E$}}
\put(-0.21,2){\makebox(0,0){$F$}}
\put(-0.21,1){\makebox(0,0){$G$}}
\put(1,-0.21){\makebox(0,0){$H$}}
\put(2.15,-0.05){\makebox(0,0){$I$}}
\put(2.17,1){\makebox(0,0){$J$}}
\end{picture}
\end{center}
\caption{A part of Grid~4.}
\label{fig_grid4}
\end{figure}
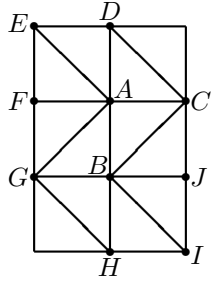
where a part of Grid~4 is shown. One type of orientation of the patches 
is represented by the patch around the node $A$ and the other one by the 
patch around the node $B$. Note that $A$ lies on an even horizontal grid 
line and $B$ on an odd horizontal grid line. 

In the previous sections, we referred to the nodes $x_i$ of the triangulation
through their indices $i$. In the following, it will be more convenient to use 
directly the notation for nodes. Thus, for example, the matrix entry $a_{ij}$
will be denoted by $a_{AB}$ if $x_i=A$ and $x_j=B$. Then the linear system
\eqref{21}, \eqref{21b}, can be written in the form
\begin{equation}\label{eq:lin_syst}
   \sum_{Q\in\II(P)}\,a_{PQ}\,u_Q=\rhs_P\quad\forall\,\,P\in\NN_h^i\,,\qquad
   u_P=u^b_P\quad\forall\,\,P\in\NN_h^b\,,
\end{equation}
where $\NN_h^i$ is the set of interior nodes of $\TT_h^{}$, $\NN_h^b$ is the
set of boundary nodes of $\TT_h^{}$, and $\II(P)\subset\NN_h^i\cup\NN_h^b$ 
consists of the node $P$ and all nodes connected to $P$ by edges of $\TT_h^{}$. 
Note that $\II(A)=\{A,B,C,D,E,F,G\}$ and $\II(B)=\{B,A,G,H,I,J,C\}$ in the case 
depicted in Fig.~\ref{fig_grid4}.

Considering the notation introduced in Fig.~\ref{fig_grid4} and the data of
Example~\ref{ex:linear}, the entries of the 
Galerkin matrix $(a_{ij})_{i,j=1}^N$ defined in \eqref{13} are given by
\begin{alignat*}{3}
&a_{AB}=a_{AD}=a_{HB}=-\varepsilon+\frac{h}6\,,\qquad &&a_{AC}=a_{FA}=a_{BJ}=a_{GB}=-\varepsilon+\frac{h}3\,,\\
&a_{BA}=a_{DA}=a_{BH}=-\varepsilon-\frac{h}6\,,\qquad &&a_{CA}=a_{AF}=a_{JB}=a_{BG}=-\varepsilon-\frac{h}3\,,\\
&a_{EA}=a_{GA}=a_{BC}=a_{BI}=\frac{h}6\,, && a_{AA}=4\,\varepsilon\,,\\
&a_{AE}=a_{AG}=a_{CB}=a_{IB}=-\frac{h}6\,, && a_{BB}=4\,\varepsilon\,,
\end{alignat*}
where $h$ is the mesh width in the directions of the coordinate axes. In our 
analytical considerations, it will be always assumed that
\begin{equation}\label{eq:eps_h}
   \varepsilon<\frac{h}9\,.
\end{equation}
Since the data of Example~\ref{ex:linear} are constant and the triangulation 
is uniform, the matrix entries do not depend on the actual position of the 
nodes $A$ and $B$. The above values of the entries of the Galerkin matrix imply 
that the relations for $P_i^\pm$ from \eqref{46} can be written in the form
\begin{equation}\label{eq:PA_PB}
   P_A^\pm=f_{AB}^\pm+f_{AC}^\pm+f_{AD}^\pm\,,\qquad
   P_B^\pm=f_{BI}^\pm+f_{BJ}^\pm+f_{BC}^\pm\,.
\end{equation}
Finally, note that, under the assumption \eqref{eq:eps_h}, one has
\begin{alignat*}{3}
&d_{AB}=d_{AD}=\varepsilon-\frac{h}6\,,\qquad &&d_{AC}=d_{AF}=\varepsilon-\frac{h}3\,,\qquad && d_{AE}=d_{AG}=-\frac{h}6\,,\\
&d_{BA}=d_{BH}=\varepsilon-\frac{h}6\,,\qquad &&d_{BG}=d_{BJ}=\varepsilon-\frac{h}3\,,\qquad && d_{BI}=d_{BC}=-\frac{h}6\,.
\end{alignat*}

A closer look at Fig.~\ref{fig:ex23} (left) reveals that, in a large part of 
the computational domain, the discrete solution $u_h$ is approximately given by
\begin{alignat}{3}
  &u_h(x,y)=x+\alpha\qquad&&\mbox{along odd horizontal grid lines}\,,
  \label{eq:odd_lines}\\
  &u_h(x,y)=x-\beta\qquad&&\mbox{along even horizontal grid lines}\,,
  \label{eq:even_lines}
\end{alignat}
where $\alpha$ and $\beta$ are positive constants. A direct computation shows
that the nodal values of this function satisfy
\begin{equation}\label{eq:perturbed_Galerkin}
   \sum_{Q\in\II(A)}\,a_{AQ}\,u_Q=h^2-2\,\delta\,\varepsilon\,,\qquad
   \sum_{Q\in\II(B)}\,a_{BQ}\,u_Q=h^2+2\,\delta\,\varepsilon\,,
\end{equation}
where $\delta=\alpha+\beta$. For the data of Example~\ref{ex:linear}, one has 
$\rhs_P=h^2$ in \eqref{eq:lin_syst} and hence one observes that
the function $u_h$ given by \eqref{eq:odd_lines} and \eqref{eq:even_lines} 
satisfies the Galerkin discretization up to the perturbation 
$2\,\delta\,\varepsilon$. This leads us to the surprising conclusion that, for 
the oscillating solution shown in Fig.~\ref{fig:ex23} (left), the AFC 
stabilization term should be nearly zero.

Thus, let us investigate the AFC stabilization term when it is applied to 
a function satisfying \eqref{eq:odd_lines} and \eqref{eq:even_lines}. If we 
consider $\delta\le h$, then
\begin{alignat}{7}
   &f_{AB}&&\le0\,,\quad &f_{AC}&\le0\,,\quad &f_{AD}&\le0\,,\quad
    &f_{AE}&\ge0\,,\quad &f_{AF}&\ge0\,,\quad &f_{AG}&\ge0\,,
   \label{eq:signs_of_fluxes_A}\\
   &f_{BI}&&\le0\,,\quad &f_{BJ}&\le0\,,\quad &f_{BC}&\le0\,,\quad 
    &f_{BA}&\ge0\,,\quad &f_{BG}&\ge0\,,\quad &f_{BH}&\ge0\,,
   \label{eq:signs_of_fluxes_B}
\end{alignat}
which together with \eqref{eq:PA_PB} implies that $P_A^+=P_B^+=0$ and hence
$R_A^+=R_B^+=1$. Furthermore, it follows from \eqref{46}, 
\eqref{eq:signs_of_fluxes_A}, and \eqref{eq:signs_of_fluxes_B} that
\begin{equation}\label{eq:QA_QB}
   Q_A^-=-f_{AE}^+-f_{AF}^+-f_{AG}^+\,,\qquad
   Q_B^-=-f_{BA}^+-f_{BG}^+-f_{BH}^+\,.
\end{equation}
Then a direct computation gives
\begin{align*}
   &P_A^-=Q_B^-=\varepsilon\,(h+2\,\delta)-\frac{h}3\,(h+\delta)\,,\\
   &Q_A^-=P_B^-=\varepsilon\,h-\frac{h}3\,(2\,h-\delta)\,.
\end{align*}
Moreover, setting
\begin{equation}\label{eq:delta}
   \delta=\frac{h}2\,\frac{h}{h-3\,\varepsilon}\,,
\end{equation}
one obtains $P_A^-=Q_A^-$ and $P_B^-=Q_B^-$, which implies that $R_A^-=R_B^-=1$.
Since the nodes $A$ and $B$ were chosen arbitrarily, one observes that, for
any function $u_h$ given by \eqref{eq:odd_lines} and \eqref{eq:even_lines} with
$\alpha+\beta$ equal to $\delta$ from \eqref{eq:delta}, the AFC stabilization
term vanishes. This shows that our conjecture was correct.

The fact that the AFC scheme with the Kuzmin limiter does not reproduce the
exact solution $u(x,y)=x$ implies that the method is not linearity preserving
or not uniquely solvable. The following lemma shows that the former possibility
holds true.

\begin{lemma}\label{lemma:lin_pres}
Let $u\in P_1(\RR^2)$ be an arbitrary first degree polynomial and let us 
consider the arrangement from Fig.~\ref{fig_grid4} and the above matrix entries
corresponding to Example~\ref{ex:linear}. Then the quantities from \eqref{46}
computed using the nodal values of $u$ satisfy
\begin{equation*}
   P_A^+\le Q_A^+\,,\qquad
   P_A^-\ge Q_A^-
\end{equation*}
and
\begin{equation*}
   P_B^+\le\frac{2\,h-3\,\varepsilon}{h-3\,\varepsilon}\,Q_B^+\,,\qquad
   P_B^-\ge\frac{2\,h-3\,\varepsilon}{h-3\,\varepsilon}\,Q_B^-\,,
\end{equation*}
where the latter inequalities are sharp. Consequently, the AFC scheme with the
Kuzmin limiter is not linearity preserving on Grid~4 when applied to
Example~\ref{ex:linear}.
\end{lemma}

\begin{proof}First consider the inequalities at the node $A$. Since the values 
of $P_A^\pm$ and $Q_A^\pm$ do not change if a constant function is added to
$u$, one can consider $u_A=0$. Moreover, the ratios $Q_A^+/P_A^+$ and
$Q_A^-/P_A^-$ do not change if $u$ is multiplied by a positive constant. Thus,
it suffices to consider three types of functions $u$: with $u_C=-u_F=1$,
$u_C=-u_F=-1$, and $u_C=u_F=0$. These functions are then determined by the
value $u_B$ and it is sufficient to consider $u_B\ge0$ in view of the 
axisymmetry of the patch $\Delta_A$. Then it is straightforward to verify that 
the inequalities at the node $A$ hold.

The inequalities at the node $B$ can be verified analogously. Equalities hold
for $u$ given by $u_A=u_B=0$, $u_G=1$ and $u_A=u_B=0$, $u_G=-1$, respectively.
For these functions, one gets 
$({\mathbb B}({\rm U}){\rm U})_B=h^2/(6h-9\varepsilon)$ and
$({\mathbb B}({\rm U}){\rm U})_B=-h^2/(6h-9\varepsilon)$, respectively, which
means that the considered method is not linearity preserving.
\end{proof}

In view of the previous lemma, it is not surprising that the exact solution of
Example~\ref{ex:linear} is not recovered by the AFC scheme with the Kuzmin
limiter on Grid~4. Nevertheless, it is rather disappointing that, for this very 
simple example,  the $H^1$ seminorm of the error cannot be reduced by 
considering finer meshes.

On the other hand, we also see from Lemma~\ref{lemma:lin_pres} that it is easy 
to modify the AFC scheme with the Kuzmin limiter in such a way that the method 
becomes linearity preserving for the considered case. In fact, similarly as in
\eqref{R-definition-BJK}, it suffices to replace $R_B^\pm$ by
\begin{equation}\label{eq:mod_rb}
   R_B^+=\min\left\{1,\frac{\mu\,Q_B^+}{P_B^+}\right\},\quad
   R_B^-=\min\left\{1,\frac{\mu\,Q_B^-}{P_B^-}\right\},
\end{equation}
with an appropriate positive constant $\mu$. It can be easily verified that
this does not change other properties of the method formulated so far. To 
simplify our analytical considerations, we shall use 
\begin{equation}\label{eq:mu}
   \mu=\frac{2\,h-3\,\varepsilon}{h-4\,\varepsilon}\,,
\end{equation}
which is a slightly larger value than suggested by Lemma~\ref{lemma:lin_pres}.
Nevertheless, for values of $h$ and $\varepsilon$ considered in our numerical
computations, this modification is negligible.

If one now repeats the computation leading to the result in Fig.~\ref{fig:ex23}
(left) with $R_i^\pm$ given by \eqref{eq:mod_rb}, \eqref{eq:mu} at nodes lying 
on odd horizontal grid lines, one obtains the exact solution $u_h(x,y)=x$. This 
is not surprising since this exact solution solves the Galerkin discretization 
and the AFC stabilization term  now vanishes for first degree polynomials. 
Thus, let us consider the following slightly more difficult example.

\begin{example}\label{ex:exp_layer}
Problem \eqref{strong-steady} is considered with $\Omega = (0,1)^2$, 
$\varepsilon=10^{-8}$, $\bb = (1,0)^T$, $c=0$, $g=1$, and
\begin{equation}\label{eq:bc_exp_layer}
   u_b(x,y)=x-\frac{{\mathrm e}^{\frac{x}\varepsilon}-1}
                   {{\mathrm e}^{\frac{1}\varepsilon}-1}\,.
\end{equation}
\end{example}

The formula in \eqref{eq:bc_exp_layer} not only defines the boundary condition
but it also represents the solution $u=u(x,y)$ of Example~\ref{ex:exp_layer}. 
In most of $\Omega$, $u(x,y)$ is very close to $x$, only in the vicinity of the 
outflow boundary $x=1$ it abruptly falls to $0$ and exhibits an exponential 
boundary layer. The approximate solution obtained on Grid~4 using the AFC 
scheme with the Kuzmin limiter modified by \eqref{eq:mod_rb}, \eqref{eq:mu} is
depicted in Fig.~\ref{fig:ex23} (right) and one can observe that it is again
rather poor. The character of the solution remains the same
also on finer meshes where one can observe that, in a large part of the 
computational domain, the discrete solution $u_h$ is approximately given by
three parameters. For example, for $ne=80$, one can deduce the following form
of the discrete solution:
\begin{alignat}{3}
  &u_h(x,y)=\left\{
        \begin{array}{ll}
                x+\alpha\quad&\mbox{if}\,\,\,x=(3\,k-1)\,h\,,\\
                x+\beta&\mbox{otherwise}\,,
        \end{array}\right.\quad&&\mbox{along odd horizontal grid lines}\,,
  \label{eq:odd_lines2}\\
  &u_h(x,y)=\left\{
        \begin{array}{ll}
                x+\beta\quad&\mbox{if}\,\,\,x=(3\,k+1)\,h\,,\\
                x-\gamma&\mbox{otherwise}\,,
        \end{array}\right.\quad&&\mbox{along even horizontal grid lines}\,,
  \label{eq:even_lines2}
\end{alignat}
where $k$ is an arbitrary integer and $\alpha$, $\beta$, $\gamma$ are 
positive constants. 

Let us now again investigate when a function $u_h$ given by 
\eqref{eq:odd_lines2}, \eqref{eq:even_lines2} satisfies the Galerkin 
discretization or the AFC scheme. Due to the definition of $u_h$, one has to 
distinguish six cases: whether the node under consideration lies on an odd or 
an even horizontal grid line and whether the respective vertical grid line is 
expressed by $x=(3\,k-1)\,h$, $x=3\,k\,h$, or $x=(3\,k+1)\,h$. Like in
\eqref{eq:perturbed_Galerkin}, one derives in all six cases that, under the
condition
\begin{equation}\label{eq:abc}
   \alpha=2\,\beta+\gamma\,,
\end{equation}
$u_h$ satisfies the Galerkin discretization up to a perturbation
$\kappa\,(\beta+\gamma)\,\varepsilon$ with $\kappa\in\{-5,-3,-1,1,2,6\}$. Since
the computed discrete solutions approximately satisfy \eqref{eq:abc}, we again 
expect that the AFC stabilization term is nearly zero for them.

To compute the AFC stabilization term for $u_h$ given by \eqref{eq:odd_lines2},
\eqref{eq:even_lines2}, we shall assume that, apart from \eqref{eq:eps_h} and
\eqref{eq:abc}, one also has
\begin{equation}\label{eq:bc}
   \beta+\gamma\le\frac{h}4\,.
\end{equation}
Then it is easy to verify that, in all cases, the fluxes again have the signs
given in \eqref{eq:signs_of_fluxes_A} and \eqref{eq:signs_of_fluxes_B}. Thus,
one again immediately obtains that $R_A^+=R_B^+=1$ and the relations
\eqref{eq:QA_QB} hold. Using \eqref{eq:PA_PB}, \eqref{eq:QA_QB}, 
\eqref{eq:mod_rb}, and \eqref{eq:mu}, a lengthy but straightforward computation
reveals that, in all three cases, $R_A^-=R_B^-=1$, which means that the AFC
stabilization term again vanishes. Note that the condition \eqref{eq:bc} allows
more flexibility in defining the function $u_h$ than \eqref{eq:delta} which
determines the respective $u_h$ uniquely up to an additive constant.

There are two important conclusions of the above discussion. The first one, a
more general, is that approximate solutions may be polluted by spurious
oscillations despite the validity of the DMP. This may happen also if the
right-hand side $g$ vanishes (in contrast to the above examples), see
Example~\ref{ex:bilinear} below. The second conclusion is that there are
oscillating functions (which may solve, e.g., a Galerkin discretization) for
which the algebraic stabilization term vanishes. This is a surprising 
observation that does not correspond to the usual experience that, in case of
an oscillating solution, a stabilization introduces an artificial diffusion in 
the discrete problem to suppress the oscillations. However, it is worth noting
that if a discretization of Example~\ref{ex:linear} on some of the Grids 1,
4--7, and 9 leads to an oscillating approximate solution of the type 
\eqref{eq:odd_lines}, \eqref{eq:even_lines}, then also residual-based 
stabilizations (see, e.g., \cite{RST08}) are not able to suppress the 
oscillations since the residual vanishes on any element of the triangulation.

Let us mention that if Example~\ref{ex:exp_layer} is solved on Grid~1 using the
AFC scheme with the Kuzmin limiter, one obtains the nodally exact solution,
except for the rightmost vertical interior grid line, see Fig.~\ref{fig:ex34}
\begin{figure}[t]
\centerline{
\includegraphics[width=0.27\textwidth]{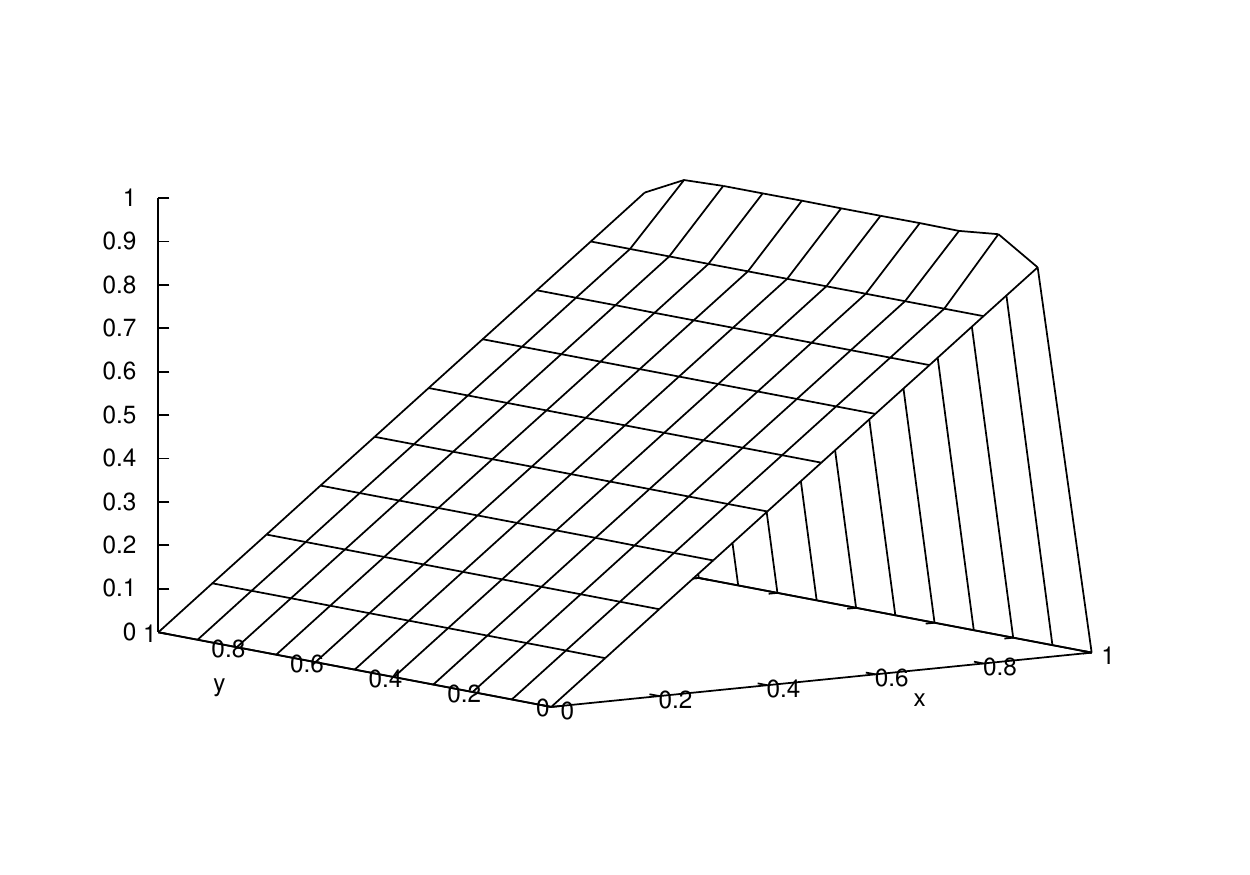}
\hspace*{-3ex}
\includegraphics[width=0.27\textwidth]{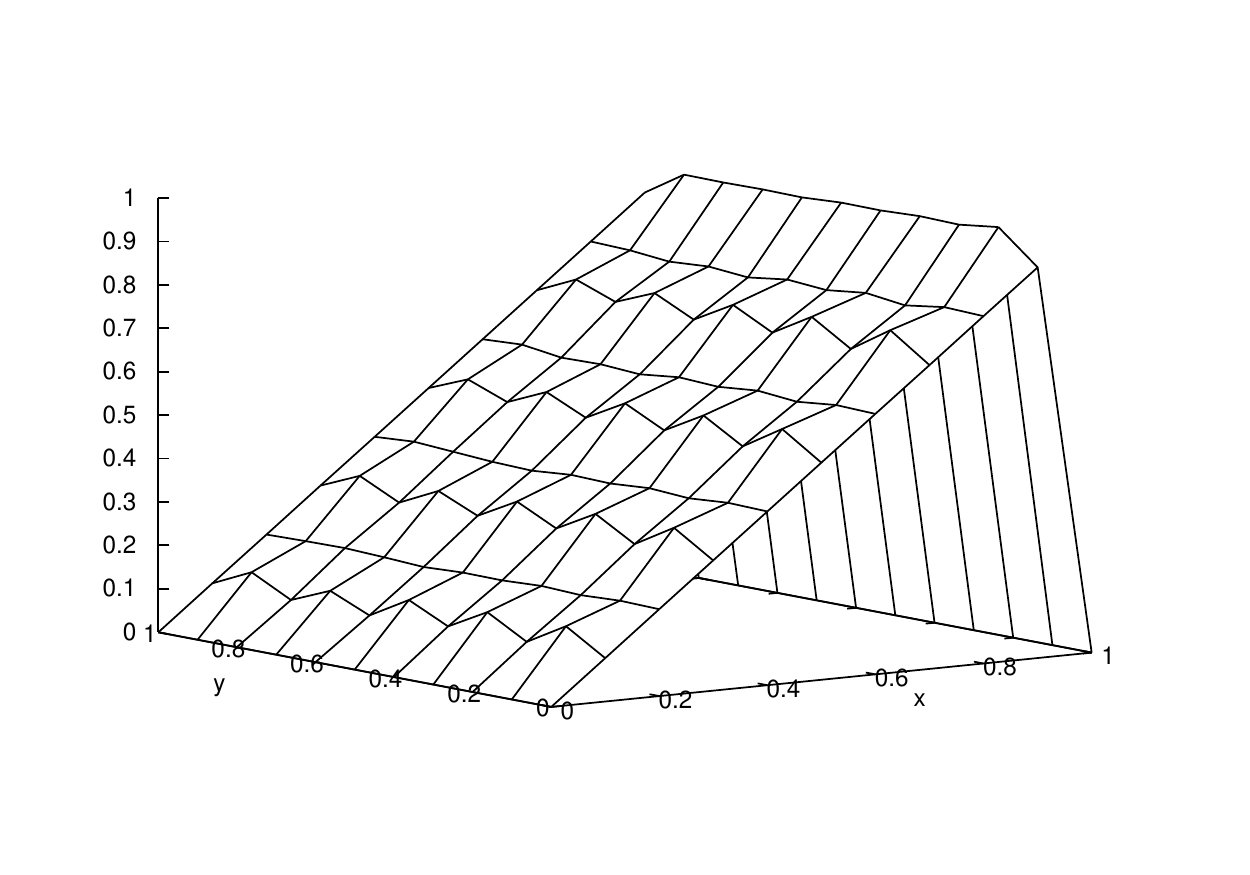}
\hspace*{-3ex}
\includegraphics[width=0.27\textwidth]{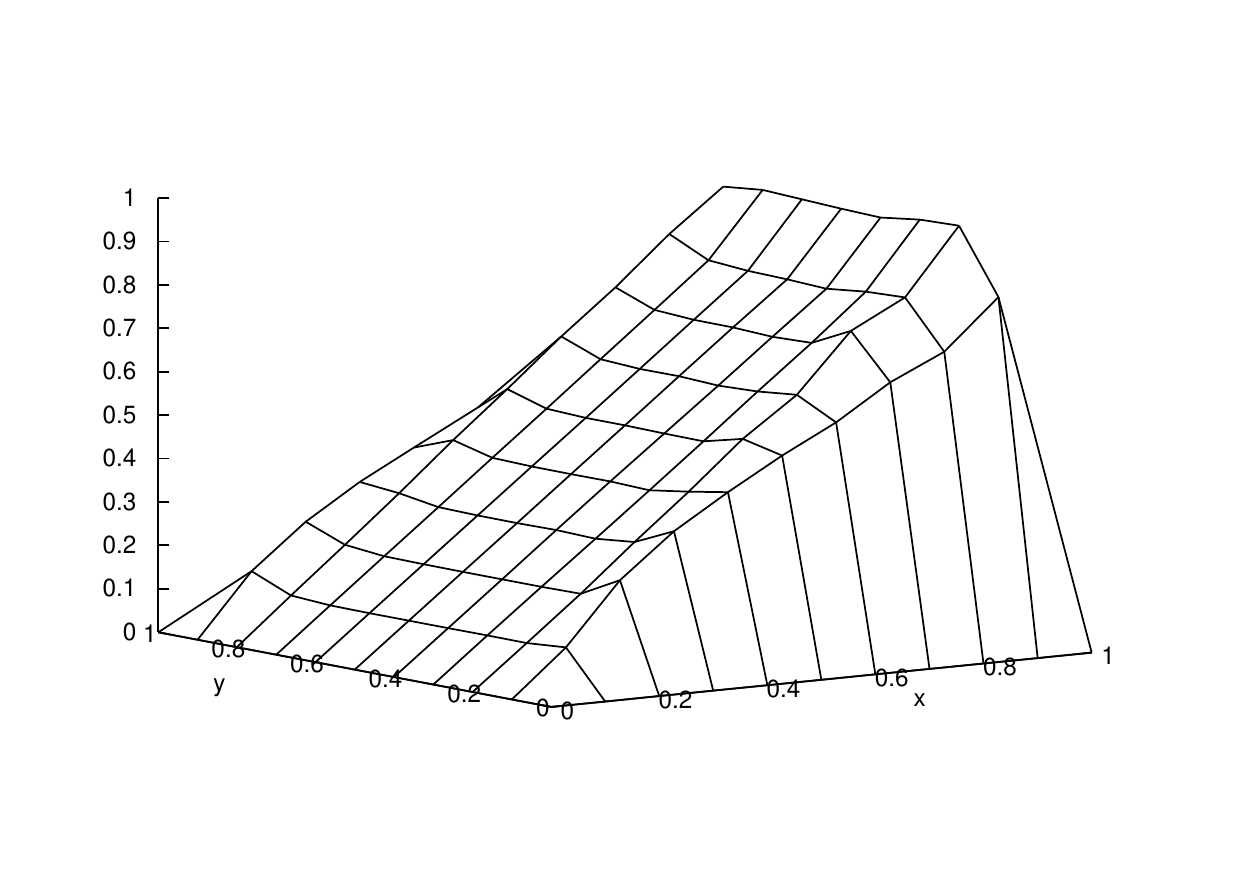}
\hspace*{-3ex}
\includegraphics[width=0.27\textwidth]{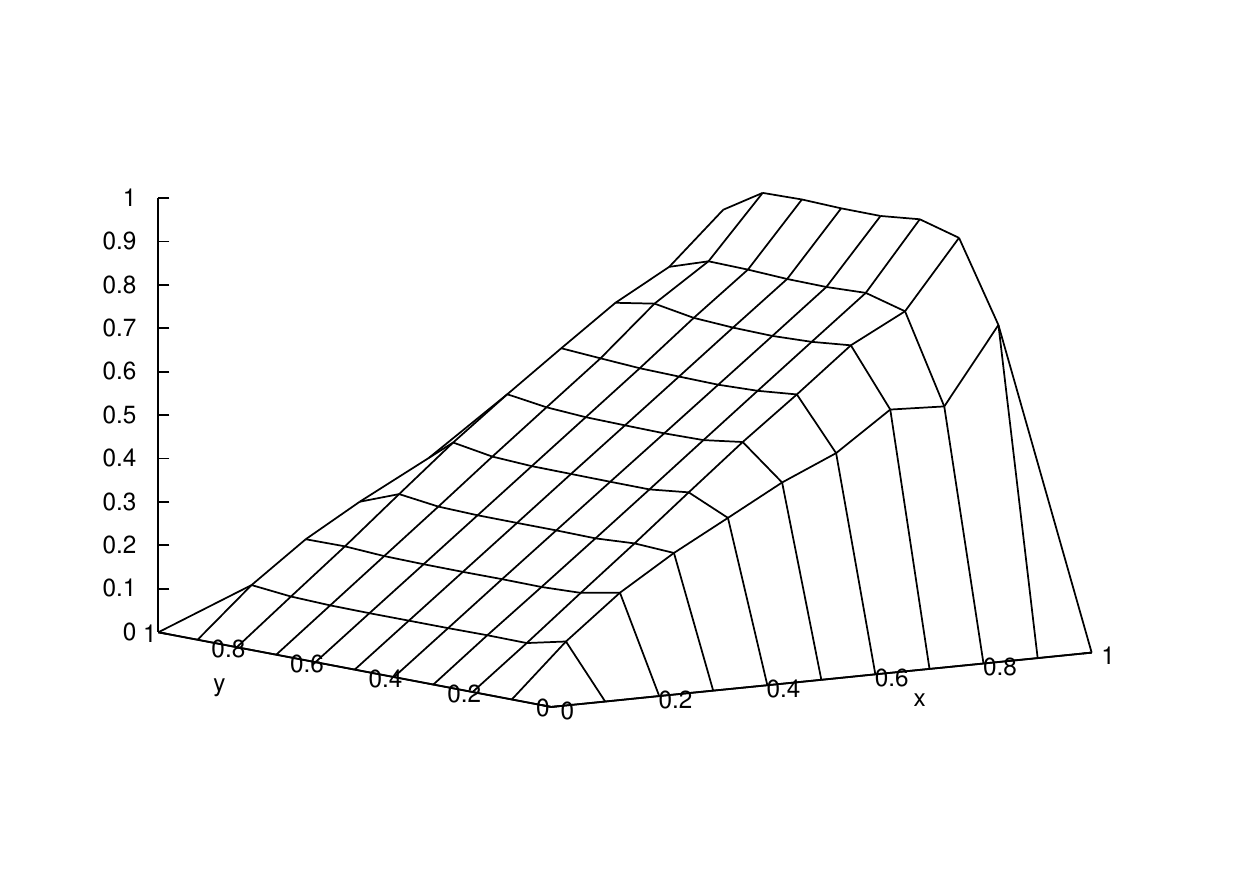}
}
\caption{Approximate solutions computed using the AFC scheme with the Kuzmin 
limiter on Grid~1 with $ne=10$: Example~\ref{ex:exp_layer} (left),
Example~\ref{ex:exp_layer}, $R_i^\pm$ defined by \eqref{R-definition-BJK} with
$\mu_i=2$ (second from left), Example~\ref{ex:exp_par_layers} (second from
right), Example~\ref{ex:exp_par_layers}, $P_i^\pm$ defined by \eqref{eq:BJK_p} 
(right)}
\label{fig:ex34}
\end{figure}
(left). However, if $R_i^\pm$ are defined by \eqref{R-definition-BJK} with
$\mu_i=2$, then oscillations again appear, see Fig.~\ref{fig:ex34} (second from
left). Since the AFC scheme with the Kuzmin limiter is linearity preserving on
Grid~1 for constant data, this shows that the symmetry of the patches and the 
linearity preservation are not sufficient for obtaining an accurate approximate 
solution.  The error at the rightmost vertical interior grid line appears 
independently of the choice of the limiter as it was proved in \cite{Kno15b} so 
that a refinement of the mesh along the outflow boundary is needed for 
enhancing the accuracy.

Let us now change the boundary condition of Example~\ref{ex:exp_layer} to the
homogeneous one, i.e., consider

\begin{example}\label{ex:exp_par_layers}
Problem \eqref{strong-steady} is considered with $\Omega = (0,1)^2$, 
$\varepsilon=10^{-8}$, $\bb = (1,0)^T$, $c=0$, $g=1$, and $u_b=0$.
\end{example}

\noindent
The solution of this example possesses not only an exponential boundary layer 
at the outflow boundary but also two parabolic boundary layers. The AFC scheme
with the Kuzmin limiter on Grid~1 provides the approximate solution shown in 
Fig.~\ref{fig:ex34} (second from right). One can observe that, in the region of
the numerical parabolic boundary layers, the approximate solution is not
monotone in the crosswind direction. This can be improved by defining $P_i^\pm$
by \eqref{eq:BJK_p} instead of \eqref{46}, see Fig.~\ref{fig:ex34} (right). In
general, this modification decreases $R_i^\pm$ so that more artificial
diffusion is introduced, which may lead to a more pronounced smearing of
layers. Then the accuracy can be again enhanced by using a finer mesh in the
boundary layer region.

Let us mention that, for the finite element functions given by
\eqref{eq:odd_lines}, \eqref{eq:even_lines} or \eqref{eq:odd_lines2}, 
\eqref{eq:even_lines2} and for the matrix entries corresponding to Grid~4 and
the data of Example~\ref{ex:linear}, the values of the Kuzmin limiter are
determined only by the quantities $R_i^-$. Since it follows from 
\eqref{eq:signs_of_fluxes_A}, \eqref{eq:signs_of_fluxes_B} that the quantities 
$P_i^-$ attain the same values for both definitions \eqref{46} and 
\eqref{eq:BJK_p}, the above analytical results remain the same also if 
$P_i^\pm$ are defined by \eqref{eq:BJK_p}. Also the result in 
Fig.~\ref{fig:ex34} (left) is not affected by computing $P_i^\pm$ using
\eqref{eq:BJK_p}.

Fig.~\ref{fig:ex4_BJK} shows results for Example~\ref{ex:exp_par_layers} 
\begin{figure}[t]
\centerline{
\includegraphics[width=0.27\textwidth]{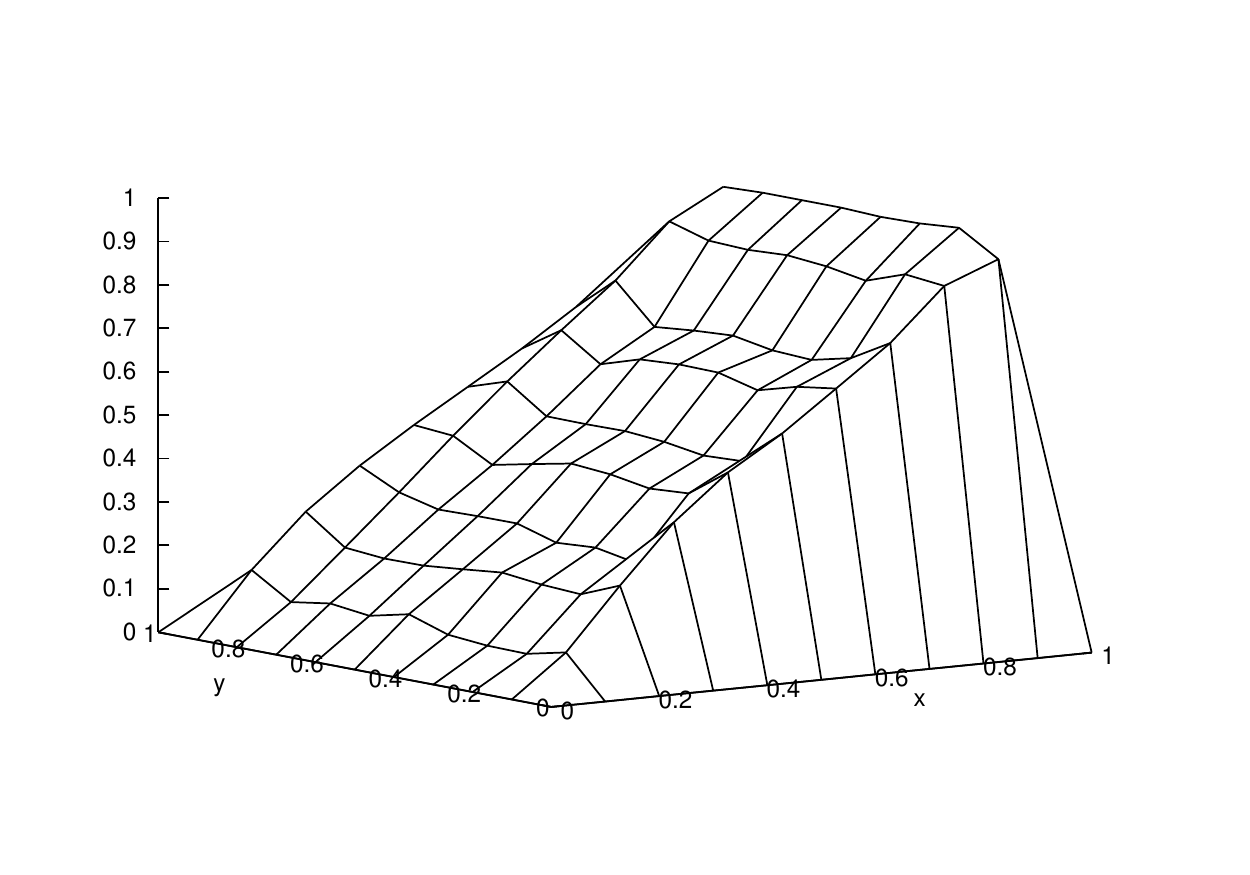}
\hspace*{2ex}
\includegraphics[width=0.27\textwidth]{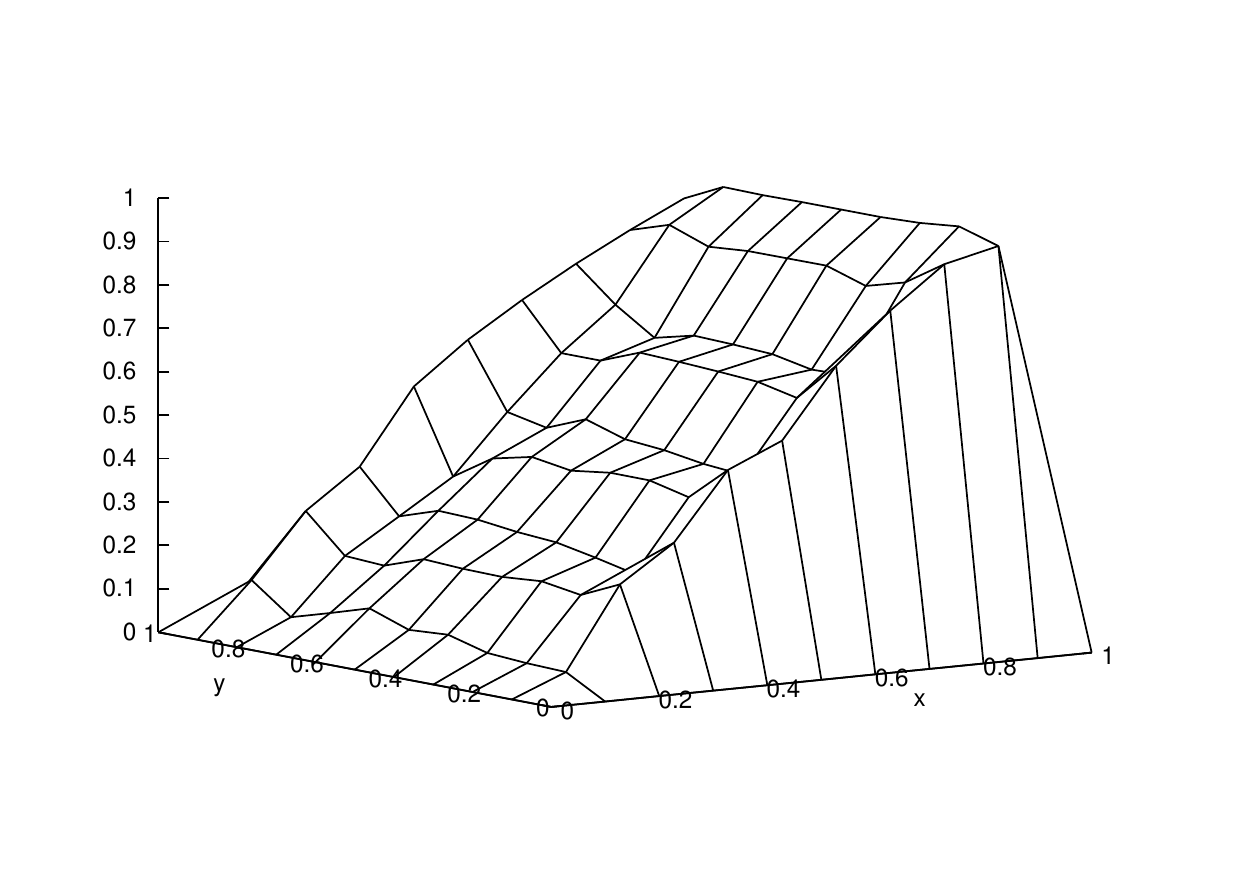}
}
\caption{Example~\ref{ex:exp_par_layers}: approximate solutions computed using 
the AFC scheme with the BJK limiter on Grid~1 with $ne=10$ for $\mu_i=1$
(left) and $\mu_i=2$ (right)}
\label{fig:ex4_BJK}
\end{figure}
computed using the AFC scheme with the BJK limiter on Grid~1. As we know from
Section~\ref{s42}, one can consider $\mu_i=1$ in \eqref{R-definition-BJK} for 
Grid~1 to guarantee the linearity preservation, which leads to the oscillatory 
solution in Fig.~\ref{fig:ex4_BJK} (left). If one uses $\mu_i=2$ as suggested 
by the formula \eqref{eq:mu_i}, the oscillations become even larger, see 
Fig.~\ref{fig:ex4_BJK} (right). This again
demonstrates that the symmetry of the patches and the linearity preservation 
are not sufficient for obtaining an accurate approximate solution. Moreover,
the results presented in Figs.~\ref{fig:ex23}, \ref{fig:ex34}, and 
\ref{fig:ex4_BJK} show that using the modification \eqref{R-definition-BJK}
(with $\mu_i>1$) of \eqref{R-definition} (e.g., to enforce the linearity 
preservation or to reduce the amount of artificial diffusion) is not a good 
idea since it allows more oscillatory solutions. In fact, this is not 
surprising since, for any finite element function, for which the quantities 
$R_i^\pm$ do not vanish, one can find $\mu_i$ such that the AFC stabilization 
term vanishes. 

In particular, one should avoid such constructions of limiters for which
$R_i^+=R_i^-=1$ may occur for oscillating functions. If $P_i^\pm$ are defined by
\eqref{eq:BJK_p} and $Q_i^\pm$ by \eqref{46}, as suggested in the discussion to 
Fig.~\ref{fig:ex34}, then $P_i^++Q_i^-=P_i^-+Q_i^+=0$ and it is easy to verify
that $R_i^+=R_i^-=1$ is equivalent to $P_i^+=Q_i^+$ and $P_i^-=Q_i^-$, i.e., to
\begin{equation*}
   0=\sum_{j\in S_i}\,(f_{ij}^++f_{ij}^-)=\sum_{j\in S_i}\,f_{ij}
    =\sum_{j\in S_i}\,d_{ij}\,(u_j-u_i)\,.
\end{equation*}
Thus, $R_i^+=R_i^-=1$ holds if and only if $u_i=\bar{u}_i$ where $\bar{u}_i$ is
a local average defined by
\begin{equation}\label{eq:average}
    \bar{u}_i=\frac{\sum_{j\in S_i}\vert d_{ij}\vert\,u_j}
    {\sum_{j\in S_i}\vert d_{ij}\vert}\,.
\end{equation}
To avoid oscillating approximate solutions, the local averages $\bar{u}_i$
should be good approximations of the values $u_i$ for smoothly varying
functions, which is the case for locally symmetric meshes like Grid~1 but not
meshes with unsymmetric patches like Grid~4. This probably contributes to the
better performance of the AFC scheme with the Kuzmin limiter on locally
symmetric meshes.

The above examples have all non-vanishing right-hand sides $g$ so that the DMP
provides only one-sided local bounds on approximate solutions. To demonstrate 
that the above-discussed phenomena are not restricted to this case, let us 
consider the following example with a vanishing right-hand side.

\begin{example}\label{ex:bilinear}
Problem \eqref{strong-steady} is considered with $\Omega = (0,1)^2$, 
$\varepsilon=10^{-8}$, $\bb = (-2,-3)^T$, $g=0$, and
\begin{equation*}
   c(x,y)=\frac{3\,x+2\,y+7}{(x+1)(y+2)}\,,\qquad
   u_b(x,y)=(x+1)(y+2)\,.
\end{equation*}
\end{example}

Note that the solution of Example~\ref{ex:bilinear} is $u(x,y)=(x+1)(y+2)$.
Whereas, on Grid~1, the AFC scheme with the Kuzmin limiter leads to an
accurate approximation, the results on Grid~4 are again polluted by spurious
oscillations, see Fig.~\ref{fig:ex5}.
\begin{figure}[t]
\centerline{
\includegraphics[width=0.55\textwidth]{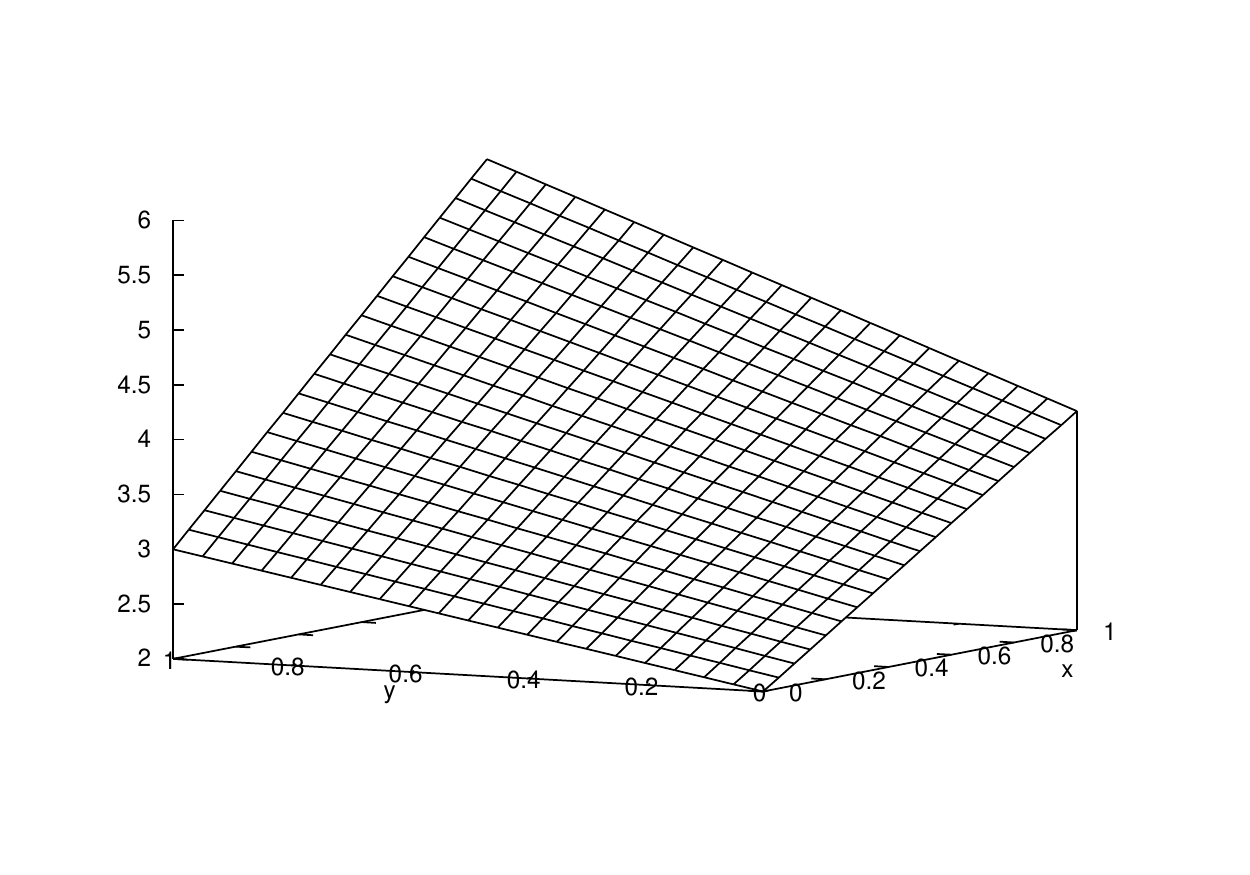}
\hspace*{-2ex}
\includegraphics[width=0.55\textwidth]{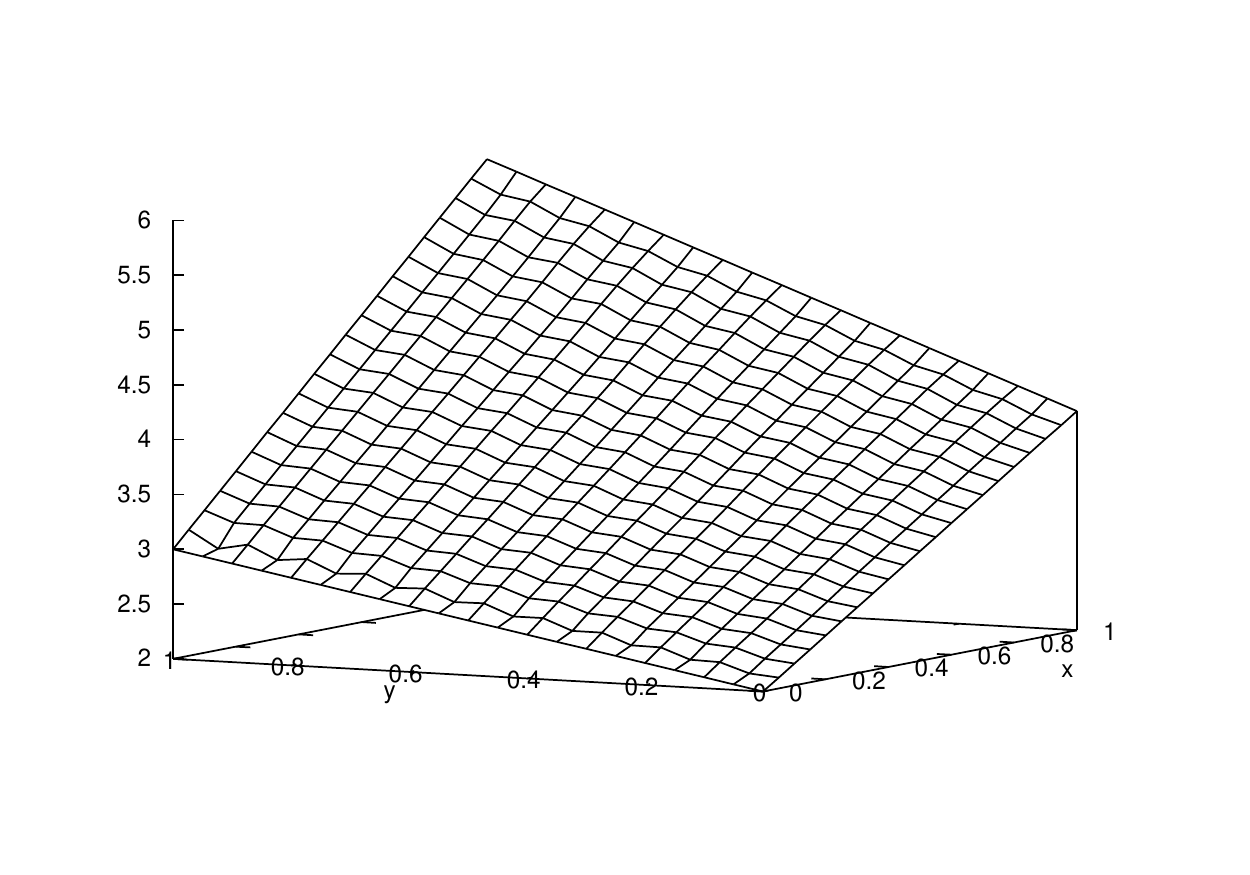}
}
\caption{Example~\ref{ex:bilinear}: approximate solutions computed using the 
AFC scheme with the Kuzmin limiter on Grid~1 (left) and on Grid~4 (right), in
both cases with $ne=20$}
\label{fig:ex5}
\end{figure}
Moreover, on Grid~1, one can observe the same optimal convergence rates as in
Table~\ref{tab1} whereas an analogous reduction of the convergence rates as in
Table~\ref{tab2} is observed on Grid~4.

\section{Symmetrized monotone upwind-type algebraically stabilized method}
\label{s6}

The aim of this section is to design an algebraic stabilization which will not
suffer from the deficiencies discussed and analyzed in the previous section. 
The starting point will be the MUAS method of Section~\ref{s43} since this
method has the favourable property being of upwind type and satisfies the DMP
on arbitrary simplicial meshes.

It was argued below Example~\ref{ex:exp_par_layers} in the previous section 
that $P_i^\pm$ should be defined by \eqref{eq:BJK_p} instead of \eqref{46}. 
Consequently, the relations \eqref{def-beta_ij_1} in the MUAS method should be 
changed to
\begin{equation}\label{51}
   P_i^+=\sum_{j\in S_i}\,\vert d_{ij}\vert\,(u_i-u_j)^+\,,\qquad\quad
   P_i^-=\sum_{j\in S_i}\,\vert d_{ij}\vert\,(u_i-u_j)^-\,.
\end{equation}
Moreover, it was observed in the previous section that two properties seem to 
be important for obtaining accurate results using algebraic stabilizations: 
local symmetries of triangulations and the linearity preservation. As it was 
demonstrated that the linearity preservation should not be enforced using 
\eqref{R-definition-BJK}, our goal will be to get this property by symmetrizing 
the definitions of $P_i^\pm$ and $Q_i^\pm$ in a suitable way.

To introduce the mentioned symmetry, we will extend the definitions of
$P_i^\pm$ and $Q_i^\pm$ by considering values at symmetrically placed points. 
\begin{figure}[t]
\begin{center}
\setlength{\unitlength}{1cm}
\begin{picture}(4.7,4.3)(-0.35,-0.1)
\thinlines

\drawline(2,2)(0,0)
\drawline(2,2)(-0.7,3.8)
\dashline{0.1}[0.001](2,2)(1.5,3)
\dashline{0.1}[0.001](2,2)(0,2)

\thicklines

\put(2,2){\circle*{0.1}}
\put(4,2){\circle*{0.1}}
\put(3,4){\circle*{0.1}}
\put(0.5,3){\circle*{0.1}}
\put(1,1){\circle*{0.1}}
\put(2.5,1){\circle*{0.1}}
\put(1.5,3){\circle*{0.1}}
\put(0,2){\circle*{0.1}}
\drawline(4,2)(3,4)(0.5,3)(1,1)(2.5,1)(4,2)
\drawline(4,2)(2,2)(3,4)
\drawline(0.5,3)(2,2)(1,1)
\drawline(2,2)(2.5,1)

\put(2.23,2.19){\makebox(0,0){$A$}}
\put(2.66,0.84){\makebox(0,0){$B$}}
\put(4.24,2){\makebox(0,0){$C$}}
\put(3.05,4.23){\makebox(0,0){$D$}}
\put(0.55,3.25){\makebox(0,0){$E$}}
\put(1.07,0.77){\makebox(0,0){$F$}}
\put(-0.23,2.05){\makebox(0,0){$\tilde C$}}
\put(1.69,3.22){\makebox(0,0){$\tilde B$}}
\end{picture}
\end{center}
\caption{Construction of symmetrically placed points}
\label{fig_symm}
\end{figure}
The construction is illustrated by Fig.~\ref{fig_symm} where a patch $\Delta_A$
around a node $A$ is shown. Each node $P$ connected to $A$ by an edge is mapped 
to a point $\tilde P$ in a symmetric way with respect to $A$ and then the idea 
is to compute the value $u_{AP}$ at $\tilde P$ of the finite element function 
$u_h$ corresponding to ${\rm U}$ via \eqref{eq_vect_fcn_identification}.
This is easy in case of the node $B$ from Fig.~\ref{fig_symm} since 
$\tilde B\in\Delta_A$. However, in case of the node $C$, the symmetrically
placed point $\tilde C$ lies outside $\Delta_A$. In this case, we extend the
linear function $u_h$ from the triangle $AEF$ to the convex set surrounded by 
the half lines $AE$ and $AF$ and define $u_{AC}$ as the value of this extended
function at $\tilde C$. This makes more sense than considering the actual value 
$u_h(\tilde C)$. The value $u_{AC}$ can be easily computed using the gradient
of $u_h$ on $AEF$ since
\begin{equation*}
   u_{AC}=u_A+\nabla u_h\vert_{AEF}^{}\cdot(\tilde C-A)
         =u_A+\nabla u_h\vert_{AEF}^{}\cdot(A-C)\,.
\end{equation*}
Of course, an analogous relation holds for $u_{AB}$, too.

Using the above-defined values $u_{AP}$, we can now symmetrize the definitions
of the quantities $P_i^\pm$ and $Q_i^\pm$ in \eqref{51} and
\eqref{def-beta_ij_2}, respectively, for any $i\in\{1,\dots,M\}$ by setting
\begin{align}
  P_i^+&=\sum_{j\in S_i}\,\vert d_{ij}\vert\,\{(u_i-u_j)^++(u_i-u_{ij})^+\}\,,
  \label{eq:mod_muas_pp}\\
  P_i^-&=\sum_{j\in S_i}\,\vert d_{ij}\vert\,\{(u_i-u_j)^-+(u_i-u_{ij})^-\}\,,
  \label{eq:mod_muas_pm}\\
  Q_i^+&=\sum_{j\in S_i}\,s_{ij}\,\{(u_j-u_i)^++(u_{ij}-u_i)^+\}\,,
  \label{eq:mod_muas_qp}\\
  Q_i^-&=\sum_{j\in S_i}\,s_{ij}\,\{(u_j-u_i)^-+(u_{ij}-u_i)^-\}\,,
  \label{eq:mod_muas_qm}
\end{align}
where
\begin{equation}\label{eq:uij}
   u_{ij}=u_i+\nabla u_h\vert_{T_{ij}}^{}\cdot(x_i-x_j)\qquad\forall\,\,j\in 
   S_i\,,
\end{equation}
and $T_{ij}\subset\Delta_i$ is a simplex intersected by the half line 
$\{x_i+\alpha\,(x_i-x_j)\,;\,\,\alpha>0\}$ (like the triangle $AEF$ in
Fig.~\ref{fig_symm} for $x_i=A$ and $x_j=C$). As we will see below, this
modification of the MUAS method leads to optimal convergence rates in cases
where the algebraic stabilizations of Section~\ref{s4} provide suboptimal
convergence results.

\pagebreak 

The above definitions of $P_i^\pm$ and $Q_i^\pm$ can be generalized to
\begin{align}
  P_i^+&=\sum_{j\in S_i,\,a_{ij}>0\,\vee\,a_{ji}>0}\,
        p_{ij}\,\{(u_i-u_j)^++(u_i-u_{ij})^+\}\,,\label{eq:smuas_pp}\\
  P_i^-&=\sum_{j\in S_i,\,a_{ij}>0\,\vee\,a_{ji}>0}\,
        p_{ij}\,\{(u_i-u_j)^-+(u_i-u_{ij})^-\}\,,\label{eq:smuas_pm}\\
  Q_i^+&=\sum_{j\in S_i}\,q_{ij}\,\{(u_j-u_i)^++(u_{ij}-u_i)^+\}\,,
                                                 \label{eq:smuas_qp}\\
  Q_i^-&=\sum_{j\in S_i}\,q_{ij}\,\{(u_j-u_i)^-+(u_{ij}-u_i)^-\}\,,
                                                 \label{eq:smuas_qm}
\end{align}
with some weighting factors satisfying, for any $j\in S_i$, $i=1,\dots,M$,
\begin{align}
   &0\le p_{ij}\le q_{ij}\,,
   \label{eq:smuas_weights}\\
   &p_{ij}>0\quad\mbox{if}\,\,\,\,a_{ij}>0\,.\label{eq:smuas_pij}
\end{align}
We name the resulting scheme Symmetrized Monotone Upwind-type Algebraically 
Stabilized (SMUAS) method. Let us recall that the stabilization matrix of the
SMUAS method is given by \eqref{asm-bij}, \eqref{asm-bii} with $\beta_{ij}$
determined by \eqref{eq:betaij}, \eqref{R-definition}, \eqref{eq:R-Dir}, and
\eqref{eq:smuas_pp}--\eqref{eq:smuas_qm} where $u_{ij}$ are defined by
\eqref{eq:uij} and $p_{ij}$, $q_{ij}$ satisfy \eqref{eq:smuas_weights},
\eqref{eq:smuas_pij}.

\begin{remark}
If $P_i^+=0$, then $R_i^+$ can be defined arbitrarily (and the same holds for
$P_i^-$ and $R_i^-$). Indeed, $P_i^+$ is used only for defining $\beta_{ij}$
with $j$ such that $u_i>u_j$. Then, if $P_i^+=0$, one has $a_{ij}\le0$ due 
to \eqref{eq:smuas_pij} and hence the matrix ${\mathbb B}({\rm U})$ defined by 
\eqref{asm-bij}, \eqref{asm-bii} does not depend on these $\beta_{ij}$.
\end{remark}

\begin{remark}\label{rem:lp}
The condition \eqref{eq:smuas_weights} assures that the SMUAS method is
linearity preserving. Indeed, if $u_h\in P_1(\mathbb{R}^d)$, then
$u_i-u_{ij}=u_j-u_i$ for any $i\in\{1,\dots,M\}$ and $j\in S_i$
and hence one gets
\begin{equation*}
   P_i^+=-P_i^-\le\sum_{j\in S_i}\,p_{ij}\,\vert u_i-u_j\vert
   \le \sum_{j\in S_i}\,q_{ij}\,\vert u_i-u_j\vert=Q_i^+=-Q_i^-\,,
\end{equation*}
so that $R_i^+=R_i^-=1$ for $i=1,\dots,N$ and the stabilization term vanishes.
\end{remark}

\begin{remark}
The condition $(a_{ij}>0\,\vee\,a_{ji}>0)$ in \eqref{eq:smuas_pp} and
\eqref{eq:smuas_pm} restricts the summation to those indices $j\in S_i$ for 
which $d_{ij}\neq0$, cf.~\eqref{eq:mod_muas_pp} and \eqref{eq:mod_muas_pm}. 
This is important to obtain optimal convergence rates in the
diffusion-dominated regime.
\end{remark}

Of course, the properties of the SMUAS method depend on the choice of the
weighting factors $p_{ij}$, $q_{ij}$. The relations
\eqref{eq:mod_muas_pp}--\eqref{eq:mod_muas_qm} correspond to
\begin{equation}\label{eq:eccomas_factors}
   p_{ij}=\max\{a_{ij},0,a_{ji}\}\,,\quad 
   q_{ij}=\max\{\vert a_{ij}\vert,a_{ji}\}\,,\quad 
   i=1,\dots,M\,,\,\,j\in S_i\,.
\end{equation}
Another possibility is to simply set
\begin{equation}\label{eq:unite_factors}
   p_{ij}=q_{ij}=1\,,\quad i=1,\dots,M\,,\,\,j\in S_i\,.
\end{equation}

More generally, let us consider weighting factors satisfying $p_{ij}=q_{ij}$ 
for $i=1,\dots,M$ and $j\in S_i$ but not necessarily equal to $1$. In the
convection-dominated regime, the condition $(a_{ij}>0\,\vee\,a_{ji}>0)$ usually
holds for any $j\in S_i$ due to the skew-symmetry of the convection matrix and
hence one obtains that $P_i^++Q_i^-=P_i^-+Q_i^+=0$. Thus, if
$R_i^+=R_i^-=1$ for some $i\in\{1,\dots,M\}$ (so that $\beta_{ij}=0$ for all
$j\in S_i$), one finds out that $u_i=\bar{u}_i$ where $\bar{u}_i$ is
a local average defined by
\begin{equation*}
    \bar{u}_i=\frac{\sum_{j\in S_i}p_{ij}\,(u_j+u_{ij})}
    {2\,\sum_{j\in S_i}p_{ij}}\,,
\end{equation*}
see the derivation leading to \eqref{eq:average}. Therefore, the choice of the
weights $p_{ij}$ may be also guided by the requirement that the local averages
$\bar{u}_i$ are good approximations of the values $u_i$ for smoothly varying
functions. Then the weights $p_{ij}$ should depend on the distribution of the
nodes $x_j$, $j\in S_i$, and on their distances to $x_i$.

\begin{remark}
It is not always necessary to use all the additional terms in
\eqref{eq:smuas_pp}--\eqref{eq:smuas_qm}. For example, let us consider the 
patch around the node $A$ in Fig.~\ref{fig_grid4}. Then the nodes $B$, $D$ and 
$C$, $F$ are symmetric with respect to $A$. Therefore, using 
\eqref{eq:unite_factors} and assuming that the condition
$(a_{ij}>0\,\vee\,a_{ji}>0)$ holds for any $j\in S_i$ and $x_i=A$, it is 
sufficient to introduce only symmetrically placed points to the nodes $E$ and 
$G$. However, for simplicity of the presentation (and also of implementation), 
we do not consider such variants of the above formulas in this paper.
\end{remark}

Now let us prove that the SMUAS method satisfies Assumptions (A1) and (A2) from
Section~\ref{s3}.

\begin{theorem}\label{smuas_a1}
The stabilization matrix of the SMUAS method satisfies Assumption~(A1).
\end{theorem}

\begin{proof}
In view of Theorem~\ref{asm_existence}, it suffices to prove \eqref{asm_beta2}.
Since $\beta_{ij}\equiv0$ for any $j\in\{1,\dots,N\}$ if $i\in\{M+1,\dots,N\}$,
consider any $i\in\{1,\dots,M\}$. Let $j\in S_i$ be such that $a_{ij}>0$. We 
want to show that $\Phi({\rm U}):=\beta_{ij}({\rm U})(u_j-u_i)$ is continuous 
at a fixed but arbitrary point 
$\bar{\rm U}=(\bar{u}_1,\dots,\bar{u}_N)\in\RR^N$. If $\bar{u}_i=\bar{u}_j$, 
then $\Phi(\bar{\rm U})=0$ and the continuity at $\bar{\rm U}$ follows from the 
estimates
\begin{equation*}
   \vert \Phi({\rm U})-\Phi(\bar{\rm U})\vert =\vert \Phi({\rm U})\vert 
   \le\vert u_i-u_j\vert\le\sqrt2\,\| {\rm U}-\bar{\rm U}\| \,,
\end{equation*}
where $\| \cdot\| $ is the Euclidean norm on ${\mathbb R}^N$. Thus, let
$\bar{u}_i>\bar{u}_j$ and denote 
\begin{equation*}
   B=\{{\rm U}\in {\mathbb R}^N;\,
   \| {\rm U}-\bar{\rm U}\| \le\frac12\vert \bar{u}_i-\bar{u}_j\vert\}\,.
\end{equation*}
Then $u_i>u_j$ for ${\rm U}\in B$ and hence
\begin{equation*}
   \Phi({\rm U})=(1-R^+_i({\rm U}))\,(u_j-u_i)\qquad\forall\,\,{\rm U}\in B\,.
\end{equation*}
Since both $P_i^+$ and $Q_i^+$ are continuous and $P_i^+$ is positive in
$B$ due to \eqref{eq:smuas_pij}, the function $\Phi$ is continuous in $B$ and 
hence also at $\bar{\rm U}$. If $\bar{u}_i<\bar{u}_j$, one proceeds analogously.
\end{proof}

\begin{theorem}\label{A2_for_ASM}
The stabilization matrix of the SMUAS method satisfies Assumption~(A2).
\end{theorem}

\begin{proof}
Consider any ${\rm U}=(u_1,\dots,u_N)\in\RR^N$, $i\in\{1,\dots,M\}$, and
$j\in S_i$. Let $u_i$ be a strict local extremum of $\rm U$ with respect to
$S_i$. We want to prove that
\begin{equation}\label{aa}
   a_{ij}+b_{ij}({\rm U})\le0\,.
\end{equation}
If $a_{ij}\le0$, then \eqref{aa} holds since $b_{ij}({\rm U})\le0$. Thus, let
$a_{ij}>0$. First, assume that $u_i>u_k$ for any $k\in S_i$. Then, for any
simplex $T\subset\Delta_i$ and any vector $\ba\in\RR^d$ pointing from $x_i$
into $T$, one has $\ba\cdot\nabla u_h\vert_T^{}<0$ with $u_h$ defined by 
\eqref{eq_vect_fcn_identification}. Thus, $u_i>u_{ik}$ for any $k\in S_i$
according to \eqref{eq:uij}, which implies that $Q_i^+=0$. Moreover,
$P_i^+\ge p_{ij}\,(u_i-u_j)^+>0$ in view of \eqref{eq:smuas_pij} and hence 
$\beta_{ij}=1-R_i^+=1$. Similarly, if $u_i<u_k$ for any $k\in S_i$, then
also $u_i<u_{ik}$ for any $k\in S_i$ and hence $Q_i^-=0$. Since 
$P_i^-\le p_{ij}\,(u_i-u_j)^-<0$, one obtains $\beta_{ij}=1-R_i^-=1$.
Therefore, in both cases, $b_{ij}({\rm U})\le -a_{ij}$, which proves \eqref{aa}.
\end{proof}

The above theorems imply that the SMUAS method is solvable
(cf.~Theorem~\ref{existence}) and satisfies the DMP formulated in
Theorem~\ref{thm:general_DMP2}. Moreover, as shown in Remark~\ref{rem:lp}, the 
SMUAS method is linearity preserving. It is important that all these properties
hold for arbitrary simplicial meshes. For regular families of triangulations,
one also has the error estimate \eqref{eq:error_est}.

However, as we have seen in Section~\ref{s5}, such theoretical properties do 
not guarantee that an algebraically stabilized method will provide an accurate
approximate solution and that the approximate solutions will converge to the 
exact solution in usual norms. Thus, let us investigate the properties of the
SMUAS method numerically. We start with Example~\ref{ex:smooth} for
$\varepsilon=10^{-8}$ and Grid~4, for which suboptimal convergence results were
presented in Table~\ref{tab2} for the AFC scheme with the Kuzmin limiter (that 
is equivalent to the MUAS method from Section~\ref{s43} in this case). The
results for the SMUAS method with $p_{ij}$, $q_{ij}$ defined by
\eqref{eq:eccomas_factors} are shown in Table~\ref{tab4}. 
\begin{table}[t]
\begin{center}
\begin{minipage}{274pt}
\caption{Example~\ref{ex:smooth}: $\varepsilon=10^{-8}$, numerical results for
Grid 4 computed using the SMUAS method with $p_{ij}$, $q_{ij}$ defined by \eqref{eq:eccomas_factors}}
\label{tab4}
\begin{tabular}{@{}rcccccc@{}}
\toprule
$ne$ & $\|u-u_h\|_{0,\Omega}^{}$ & order &   $\vert u-u_h\vert_{1,\Omega}^{}$ &
order & $\|u-u_h\|_h^{}$ & order\\
\midrule
  16 &  2.147e$-$2 &  1.61 & 4.734e$-$1 &  0.98 & 5.530e$-$2 &  1.92\\
  32 &  6.353e$-$3 &  1.76 & 2.529e$-$1 &  0.90 & 1.479e$-$2 &  1.90\\
  64 &  1.783e$-$3 &  1.83 & 1.363e$-$1 &  0.89 & 3.922e$-$3 &  1.92\\
 128 &  4.706e$-$4 &  1.92 & 7.220e$-$2 &  0.92 & 1.054e$-$3 &  1.90\\
 256 &  1.221e$-$4 &  1.95 & 3.807e$-$2 &  0.92 & 2.940e$-$4 &  1.84\\
 512 &  3.135e$-$5 &  1.96 & 2.002e$-$2 &  0.93 & 7.896e$-$5 &  1.90\\
\botrule
\end{tabular}
\end{minipage}
\end{center}
\end{table}
One observes a higher accuracy of the results than in Table~\ref{tab2} and the
experimental convergence rates tend to the optimal values. Using the SMUAS
method with $p_{ij}$, $q_{ij}$ defined by \eqref{eq:unite_factors} leads to
similar results, see Table~\ref{tab5}. 
\begin{table}[t]
\begin{center}
\begin{minipage}{274pt}
\caption{Example~\ref{ex:smooth}: $\varepsilon=10^{-8}$, numerical results for
Grid 4 computed using the SMUAS method with $p_{ij}$, $q_{ij}$ defined by \eqref{eq:unite_factors}}
\label{tab5}
\begin{tabular}{@{}rcccccc@{}}
\toprule
$ne$ & $\|u-u_h\|_{0,\Omega}^{}$ & order &   $\vert u-u_h\vert_{1,\Omega}^{}$ &
order & $\|u-u_h\|_h^{}$ & order\\
\midrule
  16 &  2.208e$-$2 &  1.60 & 4.748e$-$1 &  0.99 & 5.702e$-$2 &  1.91\\
  32 &  6.605e$-$3 &  1.74 & 2.515e$-$1 &  0.92 & 1.530e$-$2 &  1.90\\
  64 &  1.860e$-$3 &  1.83 & 1.336e$-$1 &  0.91 & 4.008e$-$3 &  1.93\\
 128 &  4.924e$-$4 &  1.92 & 6.959e$-$2 &  0.94 & 1.046e$-$3 &  1.94\\
 256 &  1.279e$-$4 &  1.95 & 3.635e$-$2 &  0.94 & 2.823e$-$4 &  1.89\\
 512 &  3.291e$-$5 &  1.96 & 1.917e$-$2 &  0.92 & 7.358e$-$5 &  1.94\\
\botrule
\end{tabular}
\end{minipage}
\end{center}
\end{table}
Also for other test examples, the
results obtained using \eqref{eq:eccomas_factors} and \eqref{eq:unite_factors}
were similar and hence we will not present any other comparisons of results for
these two choices of the weighting factors here. Similar convergence rates as
in Tables~\ref{tab4} and \ref{tab5} can be observed for all other grids from 
Figs.~\ref{fig:grids1} and \ref{fig:grids2}. The higher accuracy of the SMUAS 
method can be also seen from Fig.~\ref{fig:ex1_smuas} if one compares it with 
Fig.~\ref{fig:ex1}.

\begin{figure}[b]
\centerline{
\includegraphics[width=0.55\textwidth]{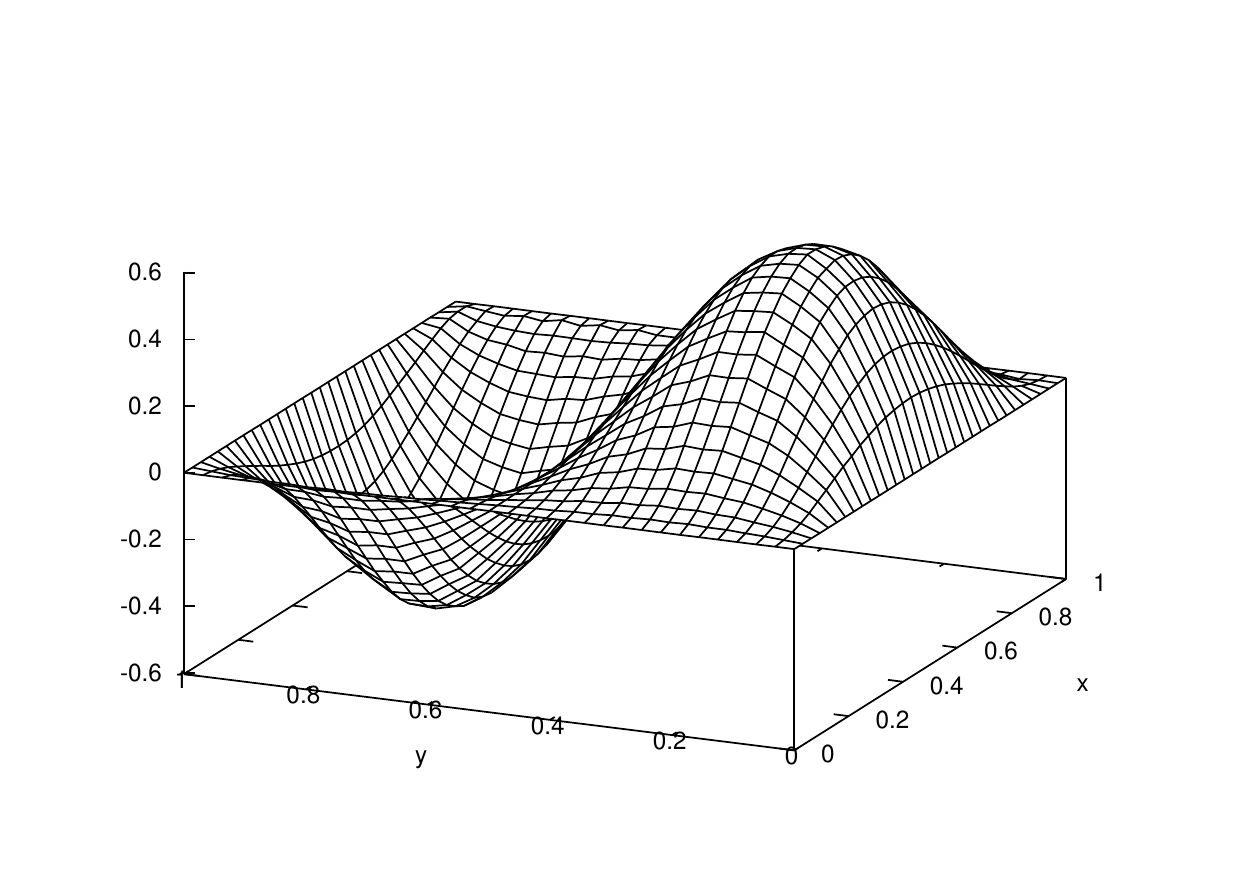}\hspace*{-2ex}
\includegraphics[width=0.55\textwidth]{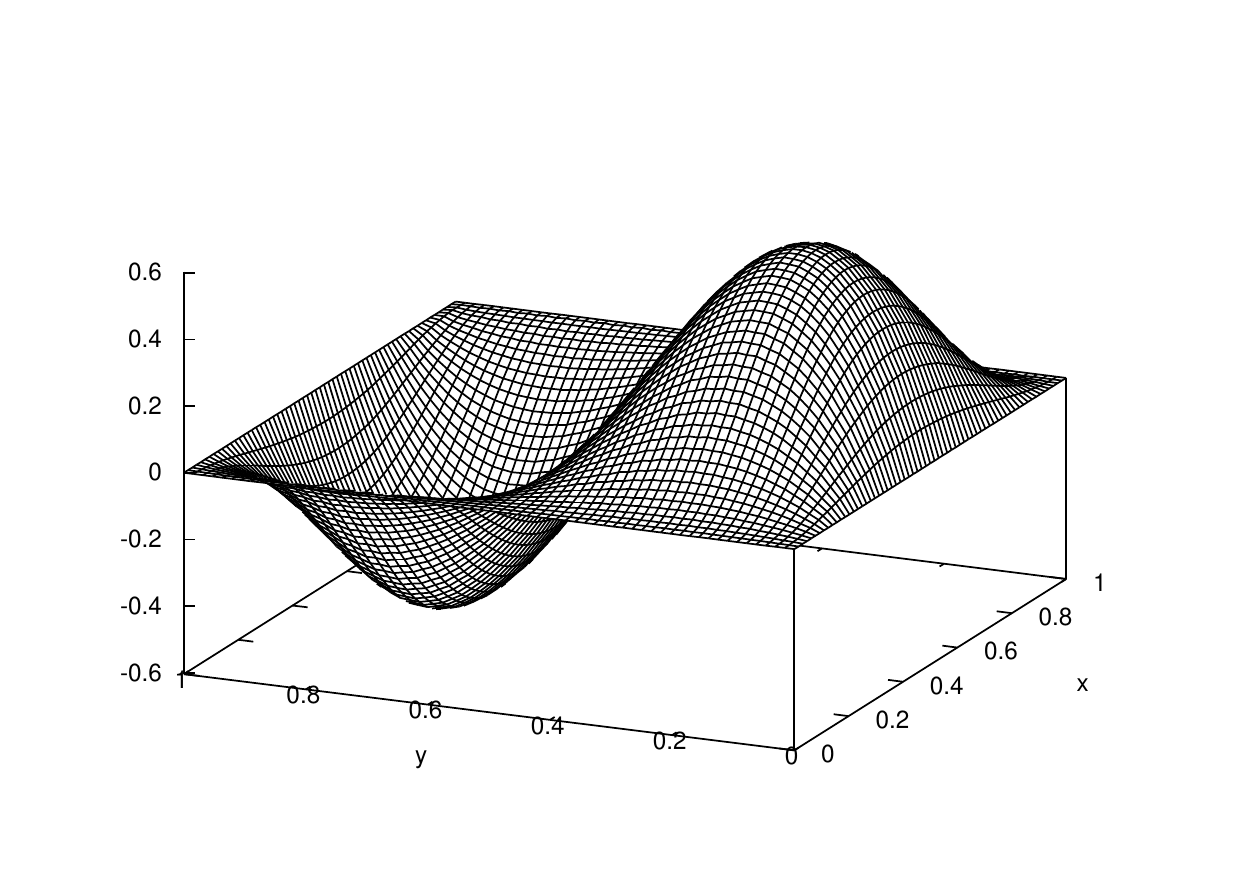}
}
\caption{Example~\ref{ex:smooth}: $\varepsilon=10^{-8}$, approximate solutions 
computed using the SMUAS method with $p_{ij}$, $q_{ij}$ defined by 
\eqref{eq:unite_factors} on Grid~4 with $ne=32$ (left) and $ne=64$ (right)}
\label{fig:ex1_smuas}
\end{figure}

It was reported in \cite{BJK16} for Example~\ref{ex:smooth} that, in the 
diffusion-dominated case $\varepsilon=10$, the solutions of the AFC scheme with 
the Kuzmin limiter do not converge for the non-Delaunay Grid~5 in any of the
three norms considered in the above tables. Recall that Grid~5 was obtained
from Grid~4 by shifting some of the nodes by $h/10$ to the right, where $h$ is 
the horizontal mesh width in Grid~4. If the shift is $h/2$, then the
experimental convergence rates tend to zero already on relatively coarse
meshes, cf.~\cite{JK21}. The MUAS method from Section~\ref{s43} shows an
improved behaviour. In particular, for the shift $h/2$, it leads to a
convergence in all three norms and the convergence rates in the $L^2$ norm and
the $H^1$ seminorm are 
near to the optimal values. However, if the shift is $0.8h$, then the accuracy 
deteriorates and the convergence rates tend to zero also for the MUAS method, 
cf.~\cite{JK21}. It is conjectured in \cite{JK21} that this behaviour is 
connected with the fact that the MUAS method is linearity preserving for the 
shift $h/2$ but not for $0.8h$. This conjecture is supported by the results 
obtained for the SMUAS method which is always linearity preserving and, indeed, 
leads to optimal convergence rates even on the highly distorted mesh 
corresponding to the shift $0.8h$, see Table~\ref{tab7}.


\begin{table}[t]
\begin{center}
\begin{minipage}{274pt}
\caption{Example~\ref{ex:smooth}: $\varepsilon=10$, numerical results computed 
using the SMUAS method with $p_{ij}$, $q_{ij}$ defined by \eqref{eq:eccomas_factors}
on triangulations of the type of Grid~5 obtained by shifting the
respective interior nodes by eight tenths of the horizontal mesh width}
\label{tab7}
\begin{tabular}{@{}rcccccc@{}}
\toprule
$ne$ & $\|u-u_h\|_{0,\Omega}^{}$ & order &   $\vert u-u_h\vert_{1,\Omega}^{}$ &
order & $\|u-u_h\|_h^{}$ & order\\
\midrule
  16 &  3.155e$-$2 &  1.65 & 5.855e$-$1 &  0.83 & 1.976e$+$0 &  0.90\\
  32 &  7.267e$-$3 &  2.12 & 3.002e$-$1 &  0.96 & 9.676e$-$1 &  1.03\\
  64 &  1.665e$-$3 &  2.13 & 1.518e$-$1 &  0.98 & 4.826e$-$1 &  1.00\\
 128 &  4.111e$-$4 &  2.02 & 7.642e$-$2 &  0.99 & 2.420e$-$1 &  1.00\\
 256 &  1.048e$-$4 &  1.97 & 3.837e$-$2 &  0.99 & 1.214e$-$1 &  1.00\\
 512 &  2.659e$-$5 &  1.98 & 1.922e$-$2 &  1.00 & 6.080e$-$2 &  1.00\\
\botrule
\end{tabular}
\end{minipage}
\end{center}
\end{table}

It is not surprising that the SMUAS method provides the exact solution on any
mesh if it is applied to Example~\ref{ex:linear}. For
Example~\ref{ex:exp_layer} and Grid~4, the solution of the SMUAS method is
nodally exact except for the rightmost vertical interior grid line, similarly 
as for the AFC scheme with the Kuzmin limiter and Grid~1 in Fig.~\ref{fig:ex34}
(left). Also for Example~\ref{ex:exp_par_layers}, the SMUAS method on Grid~4
provides an approximate solution which is nodally exact in most of the 
computational domain, see Fig.~\ref{fig:ex45_smuas} (left). The approximation 
\begin{figure}[b]
\centerline{
\includegraphics[width=0.55\textwidth]{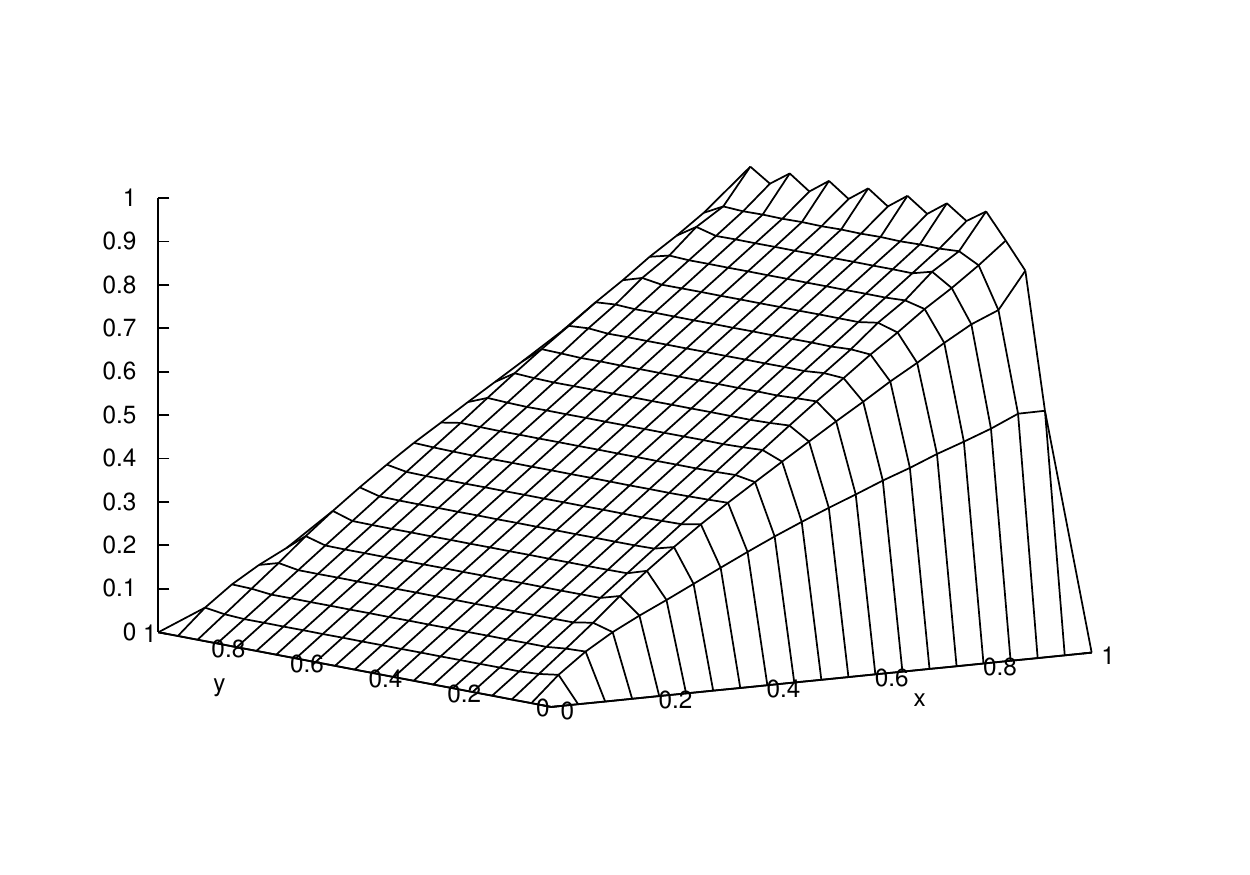}\hspace*{-2ex}
\includegraphics[width=0.55\textwidth]{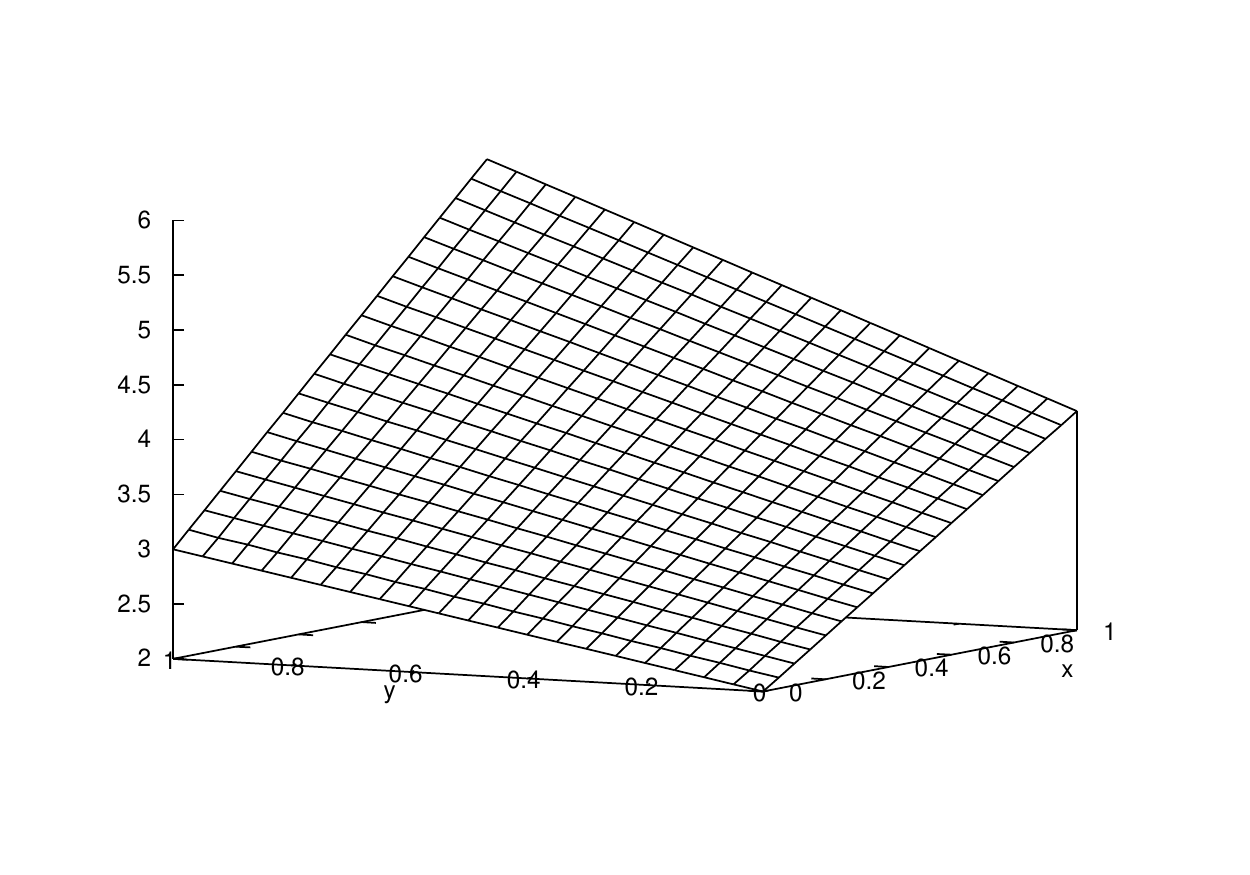}
}
\caption{Approximate solutions computed using the SMUAS method with $p_{ij}$,
$q_{ij}$ defined by \eqref{eq:eccomas_factors} on Grid~4 with $ne=20$ for 
Example~\ref{ex:exp_par_layers} (left) and Example~\ref{ex:bilinear} (right)}
\label{fig:ex45_smuas}
\end{figure}
of the boundary layers should be improved by local mesh refinement. Finally,
also in case of Example~\ref{ex:bilinear}, an application of the SMUAS method
on Grid~4 leads to a much more accurate approximate solution than the AFC
scheme with the Kuzmin limiter, see Figs.~\ref{fig:ex45_smuas} (right) and
\ref{fig:ex5} (right). Moreover, the SMUAS method again shows optimal 
convergence rates.

Summarizing our numerical results, one can state that the SMUAS method led to
optimal convergence rates in all our numerical tests involving various types of 
meshes whereas, in many cases, the algebraic stabilizations from 
Section~\ref{s4} lead to suboptimal convergence rates or do not converge at 
all. A theoretical explanation of the observed optimal convergence behaviour of 
the SMUAS method is left to future work.


\end{document}